\documentclass{amsart}
\usepackage{amssymb,amsfonts, color,epsf}
\usepackage [cmtip,arrow]{xy}
\xyoption{all}
\usepackage {pb-diagram,pb-xy}
\usepackage{enumerate}
\usepackage{accents}
\usepackage[noautoscale]{youngtab}
\usepackage{comment}

\usepackage{tikz-cd}
\makeatletter
\tikzset{
  column sep/.code=\def\pgfmatrixcolumnsep{\pgf@matrix@xscale*(#1)},
  row sep/.code   =\def\pgfmatrixrowsep{\pgf@matrix@yscale*(#1)},
  matrix xscale/.code=%
    \pgfmathsetmacro\pgf@matrix@xscale{\pgf@matrix@xscale*(#1)},
  matrix yscale/.code=%
    \pgfmathsetmacro\pgf@matrix@yscale{\pgf@matrix@yscale*(#1)},
  matrix scale/.style={/tikz/matrix xscale={#1},/tikz/matrix yscale={#1}}}
\def\pgf@matrix@xscale{1}
\def\pgf@matrix@yscale{1}
\makeatother

\ifx\pdfpageheight\undefined
   \usepackage[dvips,colorlinks=true,linkcolor=blue,citecolor=red,%
      urlcolor=green]{hyperref}
   \usepackage[dvips]{graphicx}
   \makeatletter
   \edef\Gin@extensions{\Gin@extensions,.mps}
   \makeatother
\else

  \usepackage[bookmarksopen=false,pdftex=true,breaklinks=true,%
      backref=page,pagebackref=true,plainpages=false,%
      hyperindex=true,pdfstartview=FitH,colorlinks=true,%
      pdfpagelabels=true,colorlinks=true,linkcolor=blue,%
      citecolor=red,urlcolor=green,hypertexnames=false%
      ]%
   {hyperref}
\fi
\usepackage{bm}
\usepackage{tikz-cd}

\usepackage{amsmath,amssymb,amsfonts}
\numberwithin{equation}{section}

\newtheorem{theorem}[equation]{Theorem}
\newtheorem*{theorem*}{Theorem}
\newtheorem{lemma}[equation]{Lemma}
\newtheorem{corollary}[equation]{Corollary}
\newtheorem*{corollary*}{Corollary}

\newtheorem{proposition}[equation]{Proposition}

\newtheorem{thmx}{Theorem}

\newtheorem{corollaryx}{Corollary}

\theoremstyle{definition}
\newtheorem{definition}[equation]{Definition}
\newtheorem{example}[equation]{Example}

\newtheorem{notation}[equation]{Notation}

\theoremstyle{remark}
\newtheorem{remark}[equation]{Remark}

\theoremstyle{observation}

\newcommand{\hide}[1]{}
\newcommand{\defeq}{\;{\stackrel{\text{\tiny def}}{=}}\;}

\usepackage[T1]{fontenc}
\usepackage{graphicx}
\def\chk#1{#1^{\smash{\scalebox{.7}[1.4]{\rotatebox{90}{\guilsinglleft}}}}}

\newcommand{\Ob}{\mathrm{Ob}}

\newcommand{\bbA}{{\mathbb A}}

\newcommand{\bbC}{{\mathbb C}}

\newcommand{\bbE}{{\mathbb E}}
\newcommand{\bbF}{{\mathbb F}}
\newcommand{\bbG}{{\mathbb G}}

\newcommand{\bbL}{{\mathbb L}}

\newcommand{\bbN}{{\mathbb N}}

\newcommand{\bbP}{{\mathbb P}}
\newcommand{\bbQ}{{\mathbb Q}}
\newcommand{\bbR}{{\mathbb R}}

\newcommand{\bbX}{{\mathbb X}}

\newcommand{\bbZ}{{\mathbb Z}}

\newcommand{\bH}{{\bf H}}

\newcommand{\bN}{{\bf N}}

\newcommand{\bP}{{\bf P}}
\newcommand{\bQ}{{\bf Q}}

\newcommand{\bW}{{\bf W}}
\newcommand{\bX}{{\bf X}}
\newcommand{\bY}{{\bf Y}}

\newcommand{\bj}{{\bf j}}

\newcommand{\bn}{{\bf n}}

\newcommand{\bw}{{\bf w}}
\newcommand{\bx}{{\bf x}}

\newcommand{\bz}{{\bf z}}

\newcommand{\cC}{{\mathcal C}}

\newcommand{\cE}{{\mathcal E}}
\newcommand{\cF}{{\mathcal F}}

\newcommand{\cO}{{\mathcal O}}
\newcommand{\cP}{{\mathcal P}}

\newcommand{\cT}{{\mathcal T}}

\newcommand{\rC}{{\rm C}}

\newcommand{\rE}{{\rm E}}

\newcommand{\rH}{{\rm H}}

\newcommand{\rJ}{{\rm J}}

\newcommand{\Spec}{{\rm Spec}}

\DeclareMathOperator{\HH}{H}

\DeclareMathOperator{\Proj}{Proj}

\DeclareMathOperator{\Sym}{Sym}
\DeclareMathOperator{\Gal}{Gal}

\newcommand{\isom}{\cong}
\newcommand{\Char}{\mathrm{char}}

\newcommand{\Reali}{\mathcal{R}}
\newcommand{\Trunc}{\mathrm{Trunc}}
\newcommand{\Rec}{\mathrm{Rec}}

\begin{document}
\title
[Connectivity of Joins and cohomological quantifier elimination]
{
Connectivity of joins, cohomological quantifier elimination, and an algebraic Toda's theorem
}

\author{Saugata Basu}
\address{Department of Mathematics, Purdue University,
150 N. University Street, West Lafayette, IN 47907, U.S.A.}
\email{sbasu@math.purdue.edu}
\author{Deepam Patel}
\address{Department of Mathematics, Purdue University,
150 N. University Street, West Lafayette, IN 47907, U.S.A.}
\email{patel471@purdue.edu}

\thanks{
S.B. would like to acknowledge support from the National Science Foundation award CCF-1618981, DMS-1620271,  and CCF-1910441.
D.P. would like to acknowledge support from the National Science Foundation award DMS-1502296.}

\begin{abstract}
In this article, we use cohomological techniques to obtain an algebraic version of Toda's theorem in complexity theory valid over algebraically closed fields of arbitrary characteristic. This result follows from a general `connectivity' result in cohomology. More precisely, given a closed subvariety $X \subset \bbP^{n}$ over an algebraically closed field $k$, and denoting by $\rJ^{[p]}(X) = \rJ(X,\rJ(X,\cdots,\rJ(X,X)\cdots)$ the $p$-fold iterated join of $X$ with itself, we prove that the restriction homomorphism on (singular or $\ell$-adic etale) cohomology $\rH^{i}(\bbP^{N}) \rightarrow \rH^{i}(\rJ^{[p]}(X))$, with $N = (p+1)(n+1) - 1$, is an isomorphism for 
$0 \leq i < p$, and injective for $i=p$. We also prove this result in the more general setting of relative joins for $X$ over a base scheme $S$, where $S$ is of finite type over $k$. 
We give several other applications of this connectivity result including a cohomological version of classical quantifier elimination in the 
first order theory of algebraically closed fields of arbitrary characteristic, and to obtain effective bounds on the Betti numbers of images of projective varieties under projection maps.
\end{abstract}
\maketitle

\tableofcontents

\section{Introduction}
\label{sec:intro}
The main goal of this article is to obtain a geometric proof of an algebraic version of Toda's theorem in complexity theory valid over algebraically closed fields of arbitrary characteristic. 
We obtain this result as an application of some cohomological properties of (ruled) joins of projective 
schemes. These cohomological results are also applied to obtain a cohomological version of quantifier elimination as well as give bounds for the Betti numbers of projective varieties under projection maps. We describe our main results, the motivation behind these results, 
and their connections with prior work in the following paragraphs.

\subsection{Cohomological connectivity of joins}
Let $X \subset \bbP^{m}$ and $Y \subset \bbP^{n}$ denote two 
non-empty
closed sub-schemes 
over an algebraically closed field $k$.\footnote{Here $\bbP^n$ is the usual $n$-dimensional projective space over $k$. Sometimes we denote this by $\bbP^n_k$ in order to make the base field explicit.}
Then the (ruled) join $\rJ(X,Y)$ is a closed 
subscheme
of $\bbP^{n+m+1}$. Moreover, one can show that $\rJ(X,Y)$ is connected. 
One can interpret the latter topological connectivity result as the following cohomological connectivity result:
\begin{center}
The restriction map induces an isomorphism 
$\rH^{0}(\bbP^{n+m+1}) \rightarrow \rH^{0}(\rJ(X,Y)).$
\end{center}

Our first main theorem generalizes this cohomological connectivity result to iterated joins. Given $X_{i} \subset \bbP^{n_i}$ ($0 \leq i \leq p)$, let $\rJ^{[p]}(\bbX) := \rJ(X_0,\ldots,X_p) \subset \bbP^{N}$ denote the iterated (ruled) join. Here 
$
\displaystyle{
N = \sum_{i=0}^{p}(n_i + 1) -1.
}
$

\begin{theorem*}[cf. Theorem~\ref{thm:connmultijoin}]
Let  for $0 \leq i \leq p$,
$X_i \subset \bbP^{n_i}$ be  non-empty closed 
subschemes.
Then the inclusion $\rJ^{[p]}(\bbX) \hookrightarrow \bbP^{N}$ 
(with $\displaystyle{N = \sum_{i=0}^{p}(n_i + 1) -1}$) induces an isomorphism
\begin{equation}
\rH^{j}(\bbP^{N}) \rightarrow \rH^{i}(\rJ^{[p]}(X))
\end{equation}
for all $j, 0 \leq j < p$, 
and an injective homomorphism  for $j=p$.
\end{theorem*}

The cohomology groups appearing in the Theorem may be taken to be $\ell$-adic etale cohomology with $\ell$ a fixed prime not equal to the characteristic of the base field. 
We also prove a similar result under assumptions of `higher' cohomological connectivity of the given schemes.
More precisely, we prove the following result:

\begin{theorem*}[cf. Theorem~\ref{thm:highermultijoinconn}]
Let for $0 \leq i \leq p$, $X_i \subset \bbP^{n}$ be non-empty 
closed 
subschemes, and $d_i \in \mathbb{Z}_{\geq 0}$,
such that the restriction homomorphisms  $\rH^{j}(\bbP^{n}) \rightarrow \rH^{j}(X_i)$ 
are isomorphisms for 
for $0 \leq j < d_i$, and injective for $j = d_i$.
Then the restriction homomorphism
\begin{equation}
\rH^{j}(\bbP^{N}) \rightarrow \rH^{i}(\rJ^{[p]}(X))
\end{equation}
is an isomorphism for $0 \leq j < d+p$, 
and injective for $j = d+p$, 
where $\displaystyle{
d = \sum_{i=0}^{p}d_i
}
$. 
\end{theorem*}

Note that topological connectivity properties (in the Zariski topology) of joins of projective varieties  have been
considered by various authors  (see for example the book 
\cite{Flenner-et-al}).
The main emphasis in these previous works was  on  studying Grothendieck's notion of `$d$-connectedness'. 
A projective variety $V$ is \emph{$d$-connected}  if $\dim X > d$ and $X \setminus Y$ is connected for all closed
subvarieties $Y$ of dimension $< d$. It is a classical result \cite[\S 3.2.4]{Flenner-et-al}, that
if $X$ is $d$-connected and $Y$ is $e$-connected then $J(X,Y)$ is $(d+e+1)$-connected. One can easily generalize this to the setting of multi-joins. While this result is philosophically similar to the aforementioned cohomological connectivity of the join, one cannot infer Theorem~\ref{thm:connjoin} from this result. In particular, it is easy to come up
with examples of projective varieties $X \subset \bbP^n$, such that $X$ is $d$-connected, but
the restriction homomorphism  $\HH^i(\bbP^n) \rightarrow \HH^i(X)$ is not an isomorphism
for some $i, 0 \leq i < d$. \\

The notion of cohomological connectivity considered in this paper is distinguished from Grothendieck connectivity
in another significant way.
We prove  relative versions (see Theorems~\ref{thm:relconnjoin} and \ref{thm:relconnjoin-multi}) 
of our connectedness theorems 
where the join is replaced by the relative join. 
This relative version  (namely, Theorem~\ref{thm:relconnjoin}) is in fact the key to the main applications of our connectivity theorem. 
It allows us  to relate the Poincar\'e polynomial of the image of a closed projective  scheme  with that of the iterated relative join (relative to 
the projection morphism).
\footnote{
Note that we define $P(X)(T) = \sum_{i \geq 0} \dim_{\mathbb{Q}_\ell} \HH^i(X,\mathbb{Q}_\ell)T^i$ \eqref{defn:poincarepoly}.
}
More precisely, we obtain:

\begin{theorem*} [cf. Theorem~\ref{thm:poincare}]
Let $S =  \bbP^{m}$, 
$X \subset \bbP^{n} \times \bbP^{m}$, and $\pi: \bbP^{n} \times \bbP^{m} \to \bbP^{m}$ the projection morphism.
Then,
\begin{eqnarray*}
P({\rJ^{[p]}_{S}(X)}) &\equiv&  P({\pi(X)})(1+T^{2}+ T^{4} +\cdots + T^{2((p+1)(n+1)-1)}) \  {\rm mod}  \ T^{p} .
\end{eqnarray*}
(Here, $\rJ^{([p]}_S(\cdot)$ denote the $p$-fold iterated relative join over $S$, and $P(\cdot)$ the Poincar\'e polynomial.)
\end{theorem*}

The cohomological connectivity property of the iterated relative join of a complex algebraic set     $X \subset \bbP_\bbC^m \times \bbP^n_\bbC$ relative to the proper morphism $\pi:X \rightarrow \bbP^m_\bbC$ (the restriction of the projection
$\bbP_\bbC^m \times \bbP^n_\bbC \rightarrow \bbP_\bbC^m$ to $X$)
was first investigated in \cite{Basu-Toda}. A complex version of Theorem~\ref{thm:poincare}
valid for singular cohomology was obtained there (though not stated in the language of cohomological connectivity). 
The motivation in 
\emph{loc. cit.}
was to prove an analog of  a certain result from the theory of computational complexity (Toda's theorem \cite{Toda}) in the complex algebraic setting. 
The  relation between Poincar\'e polynomials in the above theorem 
was the key input in the proof of the complex analog of 
Toda's theorem.
However, the argument in 
\emph{loc. cit.}
was topological, and heavily used  the analytic topology of complex varieties. 
Our result  extends the topological result
in 
\emph{loc. cit.}
to the setting of etale cohomology of projective schemes of finite type over a base field of arbitrary characteristic. This significantly widens the applicability of our main results.  For example, 
using our more general result we are now able to extend Toda's theorem
to algebraically closed fields of arbitrary characteristic. \\

We also give several other applications 
of our results. 
These applications are mostly quantitative in nature and impinges on 
model theory as well as on the theory of computational complexity. 
We discuss these applications in the next paragraphs.\\
 
\subsection{Cohomological quantifier elimination}
Our first application is related to the topic of `quantifier elimination' in the first order theory of algebraically closed fields.
It is a well known fact in model theory that the first order theory of algebraically closed fields (for any fixed characteristic) admits quantifier elimination. This is also known as  Chevalley's theorem. More precisely,
for $k$ an algebraically closed field 
and with tuples of variables $\bX = (X_1,\ldots, X_m), \bY = (Y_1,\ldots,Y_n)$,
a \emph{quantifier-free first order formula in the language of the field $k$}
is a Boolean formula with atoms of the
form $P(\bX,\bY) = 0, P \in k[\bX,\bY]$.
A \emph{first order formula in the language of the field $k$} is of the form
\[
\phi(\bX,\bY) = (\bQ_1 X_1)\cdots(\bQ_m X_m) \psi(\bX,\bY),
\]
where $\psi$ is a quantifier-free first order formula and each $\bQ_i$ is a quantifier belonging to $\{\exists,\forall\}$.
\footnote{We refer the reader who is unfamiliar with model theory terminology to the book \cite{Poizat} for all
the necessary background that will be required in this article.}\\

Any first order formula $\phi(\bY)$  in the language of an algebraically closed field $k$ defines (in an obvious way) a subset  $\Reali(\phi)$ of $\bbA^{n}$,  
where $n$ is the length of the tuple $\bY$.\footnote{Here, $\bbA^n$ denotes the usual $n$-dimensional affine space over $k$.} If the $n=0$ (i.e. the set of
free variables $\bY$ is empty), then the formula $\phi$ is called a \emph{sentence}, and there are only two possibilities for 
$\Reali(\phi)$. Either $\Reali(\phi) = \bbA^0$, in which case we say that $\phi$ is True (or equivalently $\phi$  belongs to the first order theory of $k$), or
$\Reali(\phi) = \emptyset$, 
in which case we say that $\phi$ is False (or $\neg\phi$ belongs to the first order theory of $k$). 
The quantifier elimination property of the theory of algebraically closed fields can now 
 be stated as:

\begin{thmx}[Quantifier-elimination in the theory of algebraic closed fields]
\label{thm:classical-qe}
Let $k$ be an algebraically closed field. Then, every first order formula
\[
\phi(\bY) = (\bQ_1 X_1)\cdots(\bQ_m X_m) \psi(\bX,\bY),
\]
in the language of the field $k$,  
there exists a quantifier-free formula
$\phi'(\bY)$ such that 
\[
\Reali(\phi) = \Reali(\phi').
\]
\end{thmx}

At the cost of being redundant (for reasons that will become apparent in the following paragraphs)  
we state the following corollary of Theorem~\ref{thm:classical-qe} in the case $\bY$ is empty.
With the same hypothesis as in Theorem~\ref{thm:classical-qe}:
\begin{corollaryx}
\label{cor:classical-qe}
\[
\phi  \Leftrightarrow (\Reali(\phi') = \bbA^0).
\]
\end{corollaryx}

We introduce in this paper a cohomological variant of quantifier elimination. 
We restrict our attention to what we call  \emph{proper}  formulas (cf. Definition \ref{def:proper-formulas}). 
Just like a first order formula  defines a constructible subset of $\bbA^n$,
a proper formula defines an algebraic subset of some products of $\bbP^n$'s. 
Given a (possibly quantified)  proper formula $\psi$ over an algebraically closed field (of arbitrary characteristic),
we produce a quantifier-free formula 
\begin{equation}
\label{eqn:J}
\psi' = J(\psi)
\end{equation}
(also proper) from $\psi$.  
(The notation $J(\cdot)$ and its connection to the join will be clear from its definition given in 
Notation \ref{not:generalized-join} in Section~\ref{sec:qe-join}.)
While not being equivalent to $\psi$ in the strict sense of model theory, 
$\psi'$ is related to
$\psi$ via a cohomological invariant (closely related to the Poincar\'e polynomial which we call the `pseudo-Poincar\'e polynomial'). \\

This invariant of 
$\psi$ can be recovered from that of the quantifier-free formula $\psi'$ using only arithmetic over $\bbZ$.
More precisely, we prove that there exists an operator $F^\omega:\bbZ[T] \rightarrow \bbZ[T]$ 
(whose definition we omit right now but can be found in \eqref{eqn:F-omega}) 
which depends only on the sequence, $\omega$, of quantifiers and the block sizes in the proper quantified formula $\psi$,
such that  the following equality holds: 

\begin{thmx}[cf. Theorem~\ref{thm:qe}]
\label{thm:cohomological-qe}
\[
Q({\psi}) =  F^\omega (Q({\psi'})).
\]
Here, $Q(\phi)$ denotes the pseudo-Poincar\'e polynomial  (see \eqref{eqn:pseudo-Poincarepolynomial3} for definition of $Q(\phi)$) of the algebraic set defined by $\phi$ for any proper formula $\phi$.
\end{thmx}

The above theorem deserves the moniker `quantifier elimination' once we substitute the realization map
$\Reali(\cdot)$, which takes formulas to constructible sets in Theorem~\ref{thm:classical-qe}, by the map $Q(\cdot)$ which takes formulas
to $\bbZ[T]$. While we have an absolute equality $\Reali(\phi) = \Reali(\phi')$ in Theorem~\ref{thm:classical-qe}, 
in Theorem~\ref{thm:cohomological-qe}, the polynomials $Q(\psi)$ and $Q(\psi')$ are related via the map $F^\omega$.
In the case of sentences (i.e. when the set of free variables is empty)  we have the  
(perhaps even more suggestive)  corollary (compare with Corollary~\ref{cor:classical-qe}):

\begin{corollaryx}
\label{cor:cohomological-qe}
\[
\psi   \Leftrightarrow (F^\omega(Q({\psi'})) = 1).
\]
\end{corollaryx}

The main advantage of the cohomological variant over usual quantifier elimination becomes apparent when
viewed through the lens of `complexity'.
In the traditional quantifier elimination (Theorem~\ref{thm:classical-qe} above)
the quantifier-free formula $\phi'$ can be potentially much more complicated than $\phi$ -- for instance, the degrees of the polynomials appearing in the atoms of $\phi'$ could be much bigger
than those of the polynomials appearing in the atoms of $\phi$ (see for example \cite{Heintz})  -- and there is no direct way of producing $\phi'$ from
$\phi$ without using algebraic constructions such as taking resultants of polynomials appearing in $\phi$ etc. (see Example~\ref{eg:complexity} below).\\

Bounding the `complexity' of the quantifier-free $\phi'$ in terms of that of  $\phi$ is an extremely well-studied question  (see for example \cite{Heintz} for the state-of-the-art)  
with many ramifications.
Indeed, the well known P vs NP question in computational complexity --
say in the  Blum-Shub-Smale (henceforth, B-S-S) model of computation \cite{BSS} -- 
is fundamentally about comparing the complexities of sequences of varieties which
belong to an `easy' class (i.e. the B-S-S complexity class $\mathrm{P}$), 
with the complexities
of sequences obtained by taking images under certain projections of sequences belonging to
the `easy' class (by taking the images under projections of sequences in the class $\mathrm{P}$
one obtains the B-S-S complexity class $\mathrm{NP}$). A formal definition
of $\mathrm{P}, \mathrm{NP}$ in the B-S-S sense can be found in \cite{Basu-Toda} and 
will not be repeated here. \\

The notion of `complexity' of a formula that we use  is made precise later (cf. Definition \ref{def:complexity}). However 
any reasonable notion of 
`complexity'  (for example, taking it to be the maximum of the degrees of polynomials that appear in it)
 suffices for the following discussion.
The best known upper bound on the `complexity' of $\phi'$ is  exponential  in that of $\phi$ \cite{Heintz}, even when the number of blocks of quantifiers is fixed and it is 
considered highly unlikely that this could be improved (see Example~\ref{eg:complexity} below).
The crucial advantage of `cohomological quantifier elimination'   over ordinary quantifier elimination (i.e. Theorem~\ref{thm:cohomological-qe} over Theorem~\ref{thm:classical-qe}) is that
the quantifier-free formula $\psi'$ has `complexity' which is bounded polynomially in that of $\psi$ (when
the length of $\omega$ is fixed).
This fact follows from the fact that $\psi'$ can be expressed in terms of $\psi$ in a uniform way -- without having to do any algebraic operations. 
Thus, while the relation between
the quantifier-free formula $\psi'$ and $\psi$ is weaker than in the case of quantifier elimination in the usual sense, it is obtained
much more easily from $\psi$ without paying the heavy price inherent in the quantifier elimination process.

\begin{example}
\label{eg:complexity}
A classical example of the blow-up in complexity on passing from $\phi$ to $\phi'$ is illustrated  in the following well-studied example. 

Let $k$ be an algebraically closed field, $V_{d,n} = \left( \Sym^d (k^{n+1})^*\right) ^{\oplus (n+1)}$, and $W_n = k^{n+1}$.  
Let $\phi_{d,n}(\mathbf{f}_0,\ldots,\mathbf{f}_n,\mathbf{X}) $ be the proper formula 
\[
(\exists \mathbf{X}) \bigwedge_{i=0}^{n} \mathbf{f}_i(\mathbf{X}) = 0,
\]
(identifying elements of $\Sym^d (k^{n+1})^*$ with  the vectors of coefficients of forms of degree $d$).
Let $X$ be the subvariety of $\bbP(V_{d,n}) \times \bbP(W_n)$ defined by
\begin{equation}
\label{eqn:eg:complexity:X}
X_{d,n} = \{([(\mathbf{f}_0,\ldots,\mathbf{f}_n)],[\mathbf{x}]) \mid \mathbf{f}_0([\mathbf{x}]) = \cdots = \mathbf{f}_n([\mathbf{x}]) = 0\},
\end{equation}
and 
\begin{equation}
\label{eqn:eg:complexity:pi}
\pi_{d,n}: \bbP(V_{d,n}) \times \bbP(W_n) \rightarrow \bbP(V_{d,n})
\end{equation}

the projection morphism.
Then the image, $\pi_{d,n}(X_{d,n})$,  is a subvariety (hypersurface)  of $\bbP(V_{d,n})$ defined by a polynomial
$R(\mathbf{f}_0,\ldots,\mathbf{f}_n)$
 (the resultant of the forms $\mathbf{f}_0,\ldots,\mathbf{f}_n$) of 
degree $(n+1)d^n$ (see for example \cite[Chapter 13, Prop. 1.1]{GKZ}). 
Notice that $\pi_{d,n}(X_{d,n}) = \Reali(\phi_{d,n})$, and in this case a quantifier-free formula $\phi'_{d,n}$ equivalent
to $\phi_{d,n}$ is given by $\phi'_{d,n} =  (R(\mathbf{f}_0,\ldots,\mathbf{f}_n)= 0)$. 
If one measures the complexity of a formula by the maximum degree of the polynomials appearing in it, we see that in this case the
complexity of $\phi_{d,n}$ is bounded by $d$, while that of $\phi'_{d,n}$ is $(n+1)d^n$ which is exponentially 
large. 
In contrast, in this case it follows from the definition of $J(\cdot)$ (Notation \ref{not:generalized-join} in Section~\ref{sec:qe-join}) that the 
complexity of the quantifier-free formula $J(\phi_{d,n})$ (cf. \eqref{eqn:J}) is bounded by $d$. 
Moreover, the operator $F^\omega$ appearing in Theorem~\ref{thm:cohomological-qe} in this 
simple example reduces to multiplication by the polynomial $(1-T)$ followed by truncation of the resulting polynomial
to degree $\dim V_{d,n} - 1$.
This illustrates the advantage of Theorem~\ref{thm:cohomological-qe} over Theorem~\ref{thm:classical-qe} from the point of 
view of complexity.
This last feature
of Theorem~\ref{thm:cohomological-qe}  is the key to our second application of Theorem~\ref{thm:poincare}  that we discuss below --
namely, an algebraic analog of Toda's theorem.  
\end{example}

We note that a version of Theorem~\ref{thm:cohomological-qe}  in a less precise form over the field of complex numbers and using singular cohomology appears in \cite{Basu-Toda}. The results of this section hold over algebraically closed fields of arbitrary characteristic, and etale cohomology and so is much more general than the result in 
\emph{loc. cit.}
Also, while the techniques used in the proof of Theorem~\ref{thm:cohomological-qe} are  somewhat similar to those used in 
\emph{loc. cit.}, the proof
differs in several key points -- so we prefer to give a self-contained proof of Theorem~\ref{thm:cohomological-qe} at the cost of some  repetition.

\subsection{Algebraic Toda's theorem}
The `cohomological quantifier elimination' theorem discussed above has applications in the theory of 
computational complexity.
In the classical theory of computational complexity, there is a clear analog of Kleene's arithmetical hierarchy in logic --
namely, the polynomial
hierarchy $\bP\bH$ (consisting of the problems of deciding sentences with
a fixed number of quantifier alternations). 
This connection, and  especially the relation to quantifiers is made precise in 
Section \ref{sec:Toda} below. 
Another important topic studied in the theory of computational complexity is the complexity of counting functions.
A particularly important class of counting functions is the class $\#\mathbf{P}$ 
(introduced by Valiant \cite{Valiant84}) 
associated with the decision problems in ${\bf NP}$: 
it can be defined as the set of functions $f(x)$ which, for any input $x$, return the number of accepting paths for the input $x$ in some non-deterministic Turing machine. 
A theorem due to Toda  relates  these two different complexity classes by an inclusion
(which expresses the fact that ability to `count' is a powerful `computational resource').
The precise result is:
\begin{theorem}[Toda \cite{Toda}]
\label{the:toda}
${\bf PH} \subset {\bf P}^{\#{\bf P}}$.
\end{theorem} 

Thus, Toda's theorem asserts that
any language in the polynomial hierarchy can
be decided by a Turing machine in polynomial time, given access to an
oracle with the power to compute  a function in $\#\mathbf{P}$.
(Only one call to the oracle is required in the proof.)
We refer the reader to \cite{Pap} for precise definitions of these classes in terms of Turing machines,
and also that of oracle computations, 
but these definitions will not be needed
for the results proved in the current paper.\\

As mentioned previously, an important feature of Theorem~\ref{thm:qe}
is that the 
quantifier-free  formula $J(\psi)$ obtained from the quantified formula $\psi$ has an easy description in terms of 
$\psi$ (in contrast to what happens in classical quantifier elimination). Making this statement quantitative leads
to a result which is formally analogous to Theorem~\ref{the:toda}, and which we discuss below. \\

As stated above Toda's theorem deals with complexity classes in a discrete setting. Blum, Shub and Smale \cite{BSS},
and independently Poizat \cite{Poizat2},
proposed a more general notion of complexity theory valid over arbitrary rings.  The classical discrete complexity theory reduces to the case when this ring is a finite field. 
The complexity classes (corresponding to the discrete complexity classes such as $\bP,\bN\bP$ etc.)
consists of sequences of \emph{constructible sets}. 

\begin{example}
\label{eg:P-NP}
For example, over any field $k$,  and 
for any fixed $d$, the sequence of algebraic sets $\left(X_{d,n}\right)_{n \geq 0}$, where $X_{d,n}$ is defined
as in \eqref{eqn:eg:complexity:X} 
will belong to the complexity class $\bP_k$.
This is because it is possible to check membership in the sets $X_{d,n}$ (by a Blum-Shub-Smale machine \cite{BSS}) with number of steps bounded by a polynomial in $n$ (for fixed $d$).
On the other hand, the sequence $\left(\pi_{d,n}(X_{d,n}\right)_{n \geq 0}$ (cf. \eqref{eqn:eg:complexity:pi})
belongs to the class $\bN\bP_k$, since its elements are obtained as images of projections of the sets
belonging to the class  $\bP_k$.
We refer the reader to Section~\ref{subsec:complexity-classes} for the precise definitions of complexity classes that we will use in this paper, but the two examples given above can be considered to be the prototypical examples of members 
of the classes $\bP_k$ and $\bN\bP_k$ respectively (also of the classes $\bP_k^c$ and $\bN\bP_k^c$
where the superscript $c$ indicates that the elements of the sequence are compact in the case 
$k=\bbR,\bbC$).
(Also note that in Section~\ref{subsec:complexity-classes}, the class $\bN\bP_k^c$ above is denoted by 
$\Sigma_k^{1,c}$ in order to place it in its right position in the polynomial hierarchy, but in this introductory section we will continue to use the more commonly used nomenclature $\bN\bP_k^c$.)
\end{example}

An interesting question that arises in this context
is whether an analog of Toda's result hold for complexity classes defined over rings other than finite fields. While the 
polynomial hierarchy has an obvious meaning in the more general B-S-S setting, the meaning of the counting class
$\#\bP$ is less clear -- boiling down to the question what does it mean to `count'  a semi-algebraic set (for B-S-S theory
over $\bbR$) or a constructible set (for B-S-S theory over $\bbC$). 
An equivalent definition of the classical (discrete) complexity class $\#\bP$ (which is more amenable to 
amenable to generalizations to an algebraic setting) is that a sequence of functions 
$\left(f_n: \{0,1\}^n \rightarrow \mathbb{Z}\right)_{n \geq 0}$ belong to the class $\#\bP$  if the functions $f_n$ count the cardinalities of the fibers of the projections maps restricted to a sequence of sets in $\bP$.
Making the reasonable choice that `counting' over $\bbR$ or $\bbC$  should mean computing the Poincar\'e polynomial, and defining the class $\#\bP$ appropriately, real and complex versions of Toda's theorem were proved in \cite{BZ09} and \cite{Basu-Toda}, respectively. \\

The proofs of the results in \cite{BZ09,Basu-Toda}  were topological and used the
euclidean topology of real and complex varieties. 
Since the approach in the current paper is purely algebraic, 
we are now able obtain a similar result in all characteristics. 
The algebraic approach is also different in certain important 
technical details.
 Additionally, in order to 
 make our result independent of the technical details which are inherent in any description of 
 a computing machine (such as B-S-S or Turing machines)
 we state and prove our result in the  non-uniform setting of circuits -- and reformulate Toda's theorem as a 
containment of two non-uniform complexity classes of \emph{constructible functions}  instead. This does not affect the main mathematical
content of the theorem, viz. a polynomially bounded  reduction of the  quantifier elimination problem in the theory of
algebraically closed fields to the problem of computing the Poincar\'e polynomial of certain algebraic set built in terms of the given formula. 
As an added advantage, this lessens the burden on the reader unfamiliar with B-S-S
machines. 
We prove the following inclusion, where the precise definitions of the complexity classes on both sides
can be found in Section~\ref{subsec:complexity-classes} and should be thought of as the non-uniform, constructible function
analogs of the classes appearing in Toda's original theorem.

\begin{theorem*}[cf. Theorem~\ref{thm:Toda-algebraic}]
\[
\mathbf{1}_{\mathbf{PH}_k^c} \subset \#\mathbf{P}^c_k.
\]
\end{theorem*}

The precise definitions are given in Section \ref{subsec:complexity-classes} below. 
The left hand side of the inclusion is the class
of sequences of characteristic functions of the algebraic analog of languages in the polynomial
hierarchy, and the right hand side is the algebraic analog of the class 
${\bf P}^{\#{\bf P}}$ as in Toda's theorem.

\begin{example}
\label{eg:sharp-P}
We will define counting complexity classes of sequences of  functions $\#\bP_k^c$ over arbitrary algebraically closed fields later in Section~\ref{subsec:complexity-classes}. But the following 
example of a sequence in $\#\bP_k^c$ over an algebraically closed field $k$  is instructive.
We use the same notation as in Examples \ref{eg:complexity} and \ref{eg:P-NP}.

The following sequence of functions is an example of a sequence in $\#\bP_k^c$:
\[
\left(f_n: \bbP(V_{d,n}) \rightarrow \bbZ[T]\right)_{n \geq 0} 
\] 
where
\begin{eqnarray*}
f_n([\mathbf{f}_0,\ldots,\mathbf{f}_n]) &=& P(\pi_{d,n}|_{X_{d,n}}^{-1}([\mathbf{f}_0,\ldots,\mathbf{f}_n]) \\
&=& P(V(\mathbf{f}_0,\ldots,\mathbf{f}_n)),
\end{eqnarray*}
$P(\cdot) = \sum_{i \geq 0} b_i(\cdot) T^i$ denotes the Poincar\'e polynomial,
and $V(\mathbf{f}_0,\ldots,\mathbf{f}_n) \subset \bbP^n_k$ is  the algebraic set defined by 
$\mathbf{f}_0=\cdots=\mathbf{f}_n = 0$.
Notice that the value of the function $f_n$ at a point in $\bbP(V_{d,n})$ is the Poincar\'e polynomial
of the fiber above the point of the map $\pi_{d,n}|_{X_{d,n}}$, and the sequence $(X_{d,n})$ belongs to the
$\bP^c_k$. In this sense the functions $f_n$ are `counting'  the fibers of projection maps restricted to a 
a sequence in $\bP^c_k$, analogous to the discrete case.\\

On the other hand the sequence $\left(\pi_{d,n}(X_{d,n})\right)_{n \geq 0}$ in Example~\ref{eg:P-NP} belongs to the class $\bN\bP^c_k$, and hence also  belongs to the class $\bP\bH^c_k$. So as an application of Theorem~\ref{thm:Toda-algebraic} we obtain that the sequence of characteristic functions 
\[
\left(1_{\pi_{d,n}(X_{d,n})}: \bbP(V_{d,n}) \rightarrow \{0,1\}  \subset \bbZ[T]\right)_{n \geq 0}
\] 
belongs to the class  $\#\bP_k^c$. \\

If Toda's original theorem expresses the `power of counting', one could say similarly, that
Theorem~\ref{thm:Toda-algebraic} is about the `expressive power of cohomology'.
\end{example}

\subsection{Uniform bounds on Betti numbers of varieties}
As a final application of our results on the connectivity of joins, we consider the well studied problem of
proving effective upper bounds on the Betti numbers of algebraic sets in terms of the parameters
defining them.  This problem has many applications, and has attracted a lot of attention in different settings. 
For example, in the
context of real algebraic and semi-algebraic sets, such bounds were first proved by
Ole{\u\i}nik and Petrovski{\u\i} \cite{OP}, 
Thom \cite{T} and
Milnor \cite{Milnor2}, who used  Morse theory and the method of counting critical points of
a Morse function to obtain a singly exponential upper bound on the Betti numbers (dimensions of the singular 
cohomology groups)  of real varieties. 
Over arbitrary fields, Katz \cite{K1}, proved similar results for the $\ell$-adic Betti numbers of both affine
and projective varieties, using prior results of Bombieri \cite{Bombieri78} and Adolphson-Sperber 
\cite{Adolphson-Sperber} on 
exponential sums. 
Theorem~\ref{thm:poincare} proved in this paper relates the Betti numbers of the image 
$\pi(X)$, of a projective subscheme $X \subset \bbP^m \times \bbP^n$, with those of $X$ itself.
Thus, it is natural to ask if this allows one to extend the results of Katz, to the images of projective
subschemes of $\bbP^m \times \bbP^n$ under projection map. One obvious way to prove upper bounds on
$\pi(X)$ is to first describe $\pi(X)$ in terms of polynomials using effective quantifier elimination
(see for example \cite{Heintz}), and then applying Katz's bound to the resulting description. 
However,
the inordinately large complexity of quantifier elimination implies that such an upper bound would be very pessimistic. \\

We utilize Theorem~\ref{thm:poincare} 
to prove uniform bounds on the Betti numbers of  the image
$\pi(X)$ of an algebraic set  $X \subset \bbP^m \times \bbP^n$ in terms of the number of equations
defining $X$ and their degrees. We are thus able to extend prior results of Katz (\cite{K1}) on bounding Betti numbers of projective algebraic sets
in terms of the number of equations defining them and their degrees, to bounding those of the image
$\pi(X)$ in terms of the same parameters.
Our main result in this direction is the following theorem.

\begin{theorem*} [cf. Theorem~\ref{thm:katz}]
Let $X \subset \bbP^{n} \times \bbP^{m}$ be an algebraic set defined by $r$ bi-homogeneous polynomials $ F_i(X_0,\ldots,X_{n},Y_0,\ldots,Y_{m})$ of bi-degree $(d_1,d_2)$, and $\pi:\bbP^{n} \times \bbP^{m} \to \bbP^{m}$ the projection morphism.
Then, for all $p > 0$,
\begin{eqnarray*}
\sum_{h=0}^{p-1} b_h(\pi(X)) &\leq & \frac{2}{p} \sum_{h=0}^{p-1} b_h(\rJ^{[p]}_\pi(X))\\
						&\leq&  \frac{2}{p} \sum_{\substack{0\leq  i \leq (n+1)(p+1) -1\\ 0 \leq j \leq m}}
						B(i+j,r(p+1),d_1+d_2).
\end{eqnarray*}
\end{theorem*}

Here, $B(n,r,d)$ is a certain function defined precisely in Section \ref{subsec:classical}, 
coming from the works of Bombieri \cite{Bombieri78}, Adolphson-Sperber \cite{Adolphson-Sperber}, and Katz 
\cite{K1},
 giving an upper bound on the 
$\ell$-adic Betti numbers (with compact support) of an algebraic subset $X \subset \bbA^N$, defined by $r$ polynomial equations
of degrees bounded by $d$. \\

An alternative method for bounding the Betti numbers of the image $\pi(X)$,
in terms of the defining parameters of $X$, is by bounding the $E_2$-terms of the spectral sequence
associated to the hypercovering of $\pi(X)$ given by the iterated products of $X$ fibered over $\pi$. We show in some
situations (Section \ref{subsec:exponential}), 
the hypercovering inequality can be loose by an exponentially large factor.
In such situations  it might be better to first express the sum of the Betti numbers of $\pi(X)$ in terms of 
certain Betti numbers of the join (cf. Eqns. \eqref{eqn:telescope-even}  and \eqref{eqn:telescope-odd}) 
and then use the bounds  due to Katz  (thus the only source of looseness of the obtained bound is that coming from 
Katz's inequality). \\

We also give an example of a situation where the join inequality can give the exact Betti numbers
(up to some dimensions) of the image $\pi(X)$. As an application of Theorem~\ref{thm:poincare} combined with a weak Lefschetz type argument, we prove the following theorem. 

\begin{theorem*}[cf. Theorem~\ref{thm:join-defect}]
Let $X \subset \bbP^N \times \bbP^n$ be a subvariety  defined by $N+r$ bi-homogeneous forms. Let $\pi: \bbP^N \times \bbP^n \rightarrow \bbP^n$ be the projection morphism. Then, for all $i, 0 \leq i < \lfloor \frac{n-r}{r} \rfloor$,
\begin{eqnarray*}
b_i(\pi(X)) &=& 1 \mbox{ if $i$ is even}, \\
b_i(\pi(X)) &=& 0 \mbox{ if $i$ is odd}.
\end{eqnarray*}
\end{theorem*}

\begin{remark}
\label{rem:join-defect}
Theorem~\ref{thm:join-defect} can be useful in determining the Betti numbers in small dimensions of varieties described as the image, $\pi(X)$, where $X \subset \bbP^a \times \bbP^b$ is a subvariety
and $\pi:\bbP^a \times \bbP^b \rightarrow \bbP^a$. In many situations, while $X$ may  be cut out by a small number
of equations, the image, $\pi(X)$,  might need many more equations to define (this number is determined by the
arithmetic rank of the elimination ideal). In particular, a repeated application of classical Artin vanishing directly to $\pi(X)$ (cf. Lemma \ref{lem:singularWL}) might not give any useful information. However, 
the interval of dimensions for which Theorem~\ref{thm:join-defect} gives us information does not depend on the arithmetic rank of the elimination ideal but just on the number of equations needed to cut out the variety $X$. In particular the result here is purely topological, and this could be useful in situations where we do not have good knowledge of the arithmetic rank of the elimination ideal.\\

An instructive example is the following. Let $m \geq n$, and let $V = k^{m \times n}$ denote the vector space of $m \times n$ matrices, and $W = k^n$. Let $X \subset \bbP(V) \times \bbP(W)$ be the subvariety defined by 
\[
X =  \{ ([A],[\mathbf{y}]) \in \bbP(V) \times \bbP(W) \mid A \mathbf{y} = \mathbf{0} \},
\]
and let
\[
\pi: \bbP(V) \times \bbP(W) \rightarrow \bbP(V)
\]
 be the projection morphism. Then $\pi(X) \subset \bbP(V)$ is the projectivization of the subvariety of $m \times n$ matrices
of rank at most $n-1$. Notice that the number of equations needed to define $X$ is clearly $\leq m$, while the number of equations needed to define $\pi(X)$ could be much larger.  In this particular situation,
the arithmetic rank of the ideal defining $\pi(X)$ is well studied,  and it is known that it is bounded from above by $mn - (n-1)^2 + 1$ \cite[Corollary 5.21]{BV} (which could be much larger than $m$). 
In this particular example, the information about the Betti numbers obtained by using Theorem~\ref{thm:join-defect} can be recovered using Lemma~\ref{lem:singularWL} directly in conjunction with the upper bound on the arithmetic rank mentioned above. However, in more general situations knowledge of a good upper bound on the arithmetic rank of the elimination ideal could be
missing, and in such situations Theorem~\ref{thm:join-defect}, whose proof is purely topological, can still give useful information. Finally, note that it is not possible to derive Theorem~\ref{thm:join-defect} from the upper bound  obtained from  the hypercover inequality. 
\end{remark}

The rest of the paper is organized as follows. In Section, \ref{sec:connectivity-of-join}, we state and prove 
our main theorems on joins and relative joins. We state and prove a key inequality (Theorem~\ref{thm:poincare})
in Section~\ref{subsec:poincare}.
In Section \ref{sec:qe-join}, we state and prove our theorem on `cohomological quantifier elimination', and in Section \ref{sec:Toda}, we give the promised application of cohomological quantifier elimination to prove a version of Toda's theorem valid over all algebraically closed fields. In Section \ref{sec:Katz}, we discuss
bounds on Betti numbers and in Section \ref{sec:comparison} we compare the efficacies of using the 
hypercovering vs the join inequalities.

\section{Cohomological connectivity properties of the join}
\label{sec:connectivity-of-join}
In this section, we prove our main result on the cohomological connectivity of the join. In the following, we shall fix an algebraically closed base field $k$ (except in subsection~\ref{subsec:non-acf}). All our schemes will be of finite type over the base field $k$.

\subsection{Joins of schemes}
We recall some basic properties of the join construction for the convenience of the reader. We refer the reader to \cite{AK} for the details. Let $S$ be a scheme of finite type over $k$. Let $\cC(S)$ denote the category of positively graded quasi-coherent $\cO_{S}$-algebras $\cT:= \bigoplus_{i = 0} ^{\infty} \cT_{i}$ such that $\cT$ is generated in degree $1$, each component is a coherent $\cO_{S}$-module, and the degree zero component is $\cO_S$. We let $\varepsilon_{\cT}: \cT \rightarrow \cO_S$ denote the corresponding projection. Given $\cT, \cP \in \cC(S)$, let $X:= \Proj(\cT)$ and $Y: = \Proj(\cP)$ denote the corresponding projective schemes. The {\it relative join} of $X$ and $Y$ over $S$, denoted $\rJ_{S}(X,Y)$, is by definition 
$\Proj(\cT \otimes_{\cO_S} \cP)$. Here are some basic properties of this construction:

\begin{enumerate}[1.]
\item
\label{itemlabel:properties:1}
 The relative join construction can be viewed as a bi-functor as follows. Any surjection $u: \cT \rightarrow \cT'$ of graded $\cO_{S}$-algebras induces a linear embedding $$P(u): \Proj(\cT') \hookrightarrow \Proj(\cT).$$ Since the tenor product is right exact, the join can be viewed as a bi-functor $\rJ_{S}(-,-): \cC(S) \times \cC(S) \rightarrow Sch_{S}$ with morphisms in $\cC(S)$ given by surjective morphisms of $\cO_S$-algebras. Here $Sch_S$ denotes the category of $S$-schemes.  
\item 
\label{itemlabel:properties:2}
Applying this construction to the the morphism $\varepsilon_{\cT} \otimes Id$, where $Id: \cP \rightarrow \cP$ is the identity, gives a natural embedding $i_X: X \hookrightarrow \Proj(\cT \otimes_{\cO_{S}} \cP) = \rJ_S(X,Y)$ of schemes over $S$. Similarly, one has a natural embedding $Y \hookrightarrow \Proj(\cT \otimes_{\cO_{S}} \cP)$.
\item 
\label{itemlabel:properties:3}
Given a morphism $S' \xrightarrow{f} S$ and an object $\cT \in \cC(S)$, let $\cT' \in \cC(S')$ denote the corresponding pull back. Since the $\Proj$ construction is compatible with base change, the relative join is also compatible with base change. In particular, one has a cartesian diagram:
\begin{equation*}
\xymatrix{
\Proj(\cT' \otimes_{\cO_{S'}} \cP') \ar[r] \ar[d] & \Proj(\cT \otimes_{\cO_S} \cP) \ar[d] \\
S' \ar[r] & S }
\end{equation*}

\end{enumerate}

We can iterate the join construction and consider the $p$-fold join $\rJ^{[p]}_{S}(X)$. More precisely, let $\rJ^{[1]}_{S}(X): = \rJ_{S}(X,X)$, and set $\rJ^{[p]}_{S}(X) : = \rJ_{S}(\rJ_S^{[p-1]}(X),X)$. 
This construction is the same as $\rJ_{S}(\underbrace{X,\cdots,X}_{p+1})$. Note that a surjection $\cP \rightarrow \cT \in \cC(S)$ induces
an imbedding $\rJ^{[p]}_{S}(X) \hookrightarrow \rJ^{[p]}_{S}(Y)$ for all $p$. \\

More generally, given $\cP_{1},\ldots, \cP_j \in \cC(S)$, we can consider the { \it multi-join}:
$$\rJ_S(\cP_1,\cdots,\cP_{j}) := \Proj(\cP_{1} \otimes \cdots \otimes \cP_j).$$\\

As before, one has closed embeddings $\Proj(\cP_i) \hookrightarrow \rJ_S(\cP_1,\cdots,\cP_{j})$. \\

Suppose $\cE$ is a vector bundle on $S$ and $X$ is a closed sub-scheme of $\bbP(\cE)$. Recall, $\bbP(\cE)$ is $\Proj$ of the symmetric algebra $\Sym^{\cdot}_{\cO_{S}}(\cE^{\vee})$, where $\cE^{\vee}$ is the dual bundle. In this case, $X$ is given by applying the $\Proj$ construction to an object $\cF$ in $\cC(S)$. More precisely, $\cF$ is a quotient of $\Sym^{\cdot}_{\cO_{S}}(\cE^{\vee})$. In particular, we have a natural embedding $\rJ^{[p]}_{S}(X) \hookrightarrow \bbP(\cE^ {\oplus (p+1)})$. We note that the construction of  $\rJ^{[p]}_{S}(X)$ depends on $\cF$ and, in particular, on the embedding of $X$ in $\bbP(\cE)$.\\

We can generalize the previous paragraph to the setting of multi-joins. 
Suppose $\cE_{i}$ ($0 \leq i \leq p$) are vector bundles on $S$, and $X_i \subset \bbP(\cE_i)$ are closed subschemes. Then each $X_i = \Proj(\cF_i)$, and we can define the multi-join $\rJ_S(X_0,\cdots,X_p)$ as before. Note that the previous constructions give a natural embedding
$$\rJ_S(X_0,\cdots,X_p) \hookrightarrow \bbP(\bigoplus_{i=0}^{p} \cE_i).$$
Given $X_0,\cdots,X_p$ as above, we shall denote the multiple join $\rJ_S(X_0,\cdots,X_p) $  by
$\rJ_S(\bbX)$. \\

Let $X_i \hookrightarrow \bbP(\cE_i)$ as above, $\pi_i: X_i \rightarrow S$ denote the structure map, and $\pi_i(X_i)$ denote the corresponding scheme theoretic image. Note that, since $\pi_i$ is proper, the underlying set of $\pi_i(X)$ is the set theoretic image. Let $\bbE:= \bigoplus_{i=0}^{p}\cE_i$, and let $\pi(\bbX)$ denote the union of the subschemes $\pi_{i}(X_i)$. Consider the base change diagram:
$$
\xymatrix{
\bbP(\bbE)_{\pi(\bbX)} \ar[r] \ar[d]  & \bbP(\bbE) \ar[d] \\
\pi(\bbX) \ar[r] & S.}
$$

\begin{lemma}
With notation as above, the structure map $\rJ_S(\bbX) \rightarrow S$ factors through $\pi(\bbX)$.
\end{lemma}
\begin{proof}
One can proceed by induction on $p$. Suppose $p =1$. Then
by (\cite{AK}, B.3),  there is a natural retraction $\rJ_S(X_0,X_1) \setminus X_0 \rightarrow X_1$ (i.e. a section of the natural embedding $X_1 \hookrightarrow \rJ_S(X_0,X_1)$).  It follows that the image of $\rJ(X_0,X_1) \setminus X_0$ in $S$ is contained in $\pi_{1}(X_1)$ and similarly for $X_0$. This proves the result in the case that $p=1$. The general case follows by induction.
\end{proof}

As a consequence of the previous lemma, and the universal property of fiber products, one has a commutative diagram:
\begin{equation*}
\xymatrix{
 \rJ_{S}(\bbX) \ar[r] \ar[d]^{q_1}  & \bbP(\bbE)_{\pi(\bbX)} \isom \bbP(\bbE|_{\pi(\bbX)}) \ar[d]^{q_2} \\
\pi(\bbX) \ar@{=}[r] & \pi(\bbX).}
\end{equation*}

In the case of $X \subset \bbP(\cE)$ and $\rJ^{[p]}_S(X) \subset \bbP(\cE^ {\oplus (p+1)})$, we get a commutative diagram:
\begin{equation*}
\xymatrix{
\rJ^{[p]}_S(X)  \ar[r] \ar[d]^{q_1}  & \bbP(\cE^ {\oplus (p+1)}) \ar[d]^{q_2} \\
\pi(X) \ar@{=}[r] & \pi(X).}
\end{equation*}

\begin{remark}
Note that $\bbP(\cE^ {\oplus (p+1)})_{\pi(X)}$ is canonically isomorphic to $\bbP(\cF^ {\oplus (p+1)})$ where 
$\cF = \cE|_{\pi(X)}$ is the restricted bundle.
\end{remark}
\begin{remark}
If $\cE$ is the trivial bundle of rank $n+1$, then we may identify $\bbP(\cE^ {\oplus (p+1)})$ with $\bbP_{S}^{(p+1)(n+1)-1}$. 
\end{remark}

\subsection{Joins and cones}
Suppose now that $S = \Spec(k)$. In the following, we shall drop the subscript $S$ from our notation in the setting of $S = \Spec(k)$ (unless we need to specify the field). Let $X \subset \bbP^{n}$ and $Y \subset \bbP^{m}$ denote two fixed projective subschemes. If $A$ (resp. $B$) is the homogeneous coordinate ring of $X$ (resp. $Y$), then we defined join of $X$ and $Y$ as the projective scheme $\rJ(X,Y) : = \Proj(A \otimes_k B)$. Note that this is naturally a closed subscheme of $\bbP^{n+m+1}$. \\
 
 The cone of $X$, denoted by $\rC(X)$, is by definition the affine scheme $\Spec(A) \subset \bbA^{n+1}$. We shall denote by $o_{X} \in \rC(X)$ the cone point. In the following, we shall sometimes drop the subscript and simply denote by $o$ the cone point. One has a canonical isomorphism:
\begin{equation}\label{joincone} 
\rC(\rJ(X,Y)) \cong \rC(X) \times_k \rC(Y). 
\end{equation}

\subsection{Proofs of cohomological connectivity of joins}
In the following, we shall prove connectivity (i.e. cohomology vanishing) results for iterated relative joins. By cohomology, we shall mean etale cohomology theory on the category of schemes over $k$. Moreover precisely, given a prime $\ell$ not equal to the characteristic of the field $k$, we shall consider the etale cohomology groups $\rH^i_{et}(X,\bbZ/\ell^n\bbZ)$, $\rH^{i}_{et}(X, \bbZ_{\ell})$ or  $\rH^{i}_{et}(X, \bbQ_{\ell})$. We shall usually drop the coefficients (and subscript) and denote these simply by $\rH^i(X)$. \\

\begin{remark}
We remind the reader that by definition $$\rH^{i}_{et}(X, \bbZ_{\ell}) := \varprojlim_n \rH^i_{et}(X,\bbZ/\ell^n\bbZ)$$ and  $\rH^{i}_{et}(X, \bbQ_{\ell}) = \rH^{i}_{et}(X, \bbZ_{\ell}) \otimes_{\bbZ_{\ell}} \bbQ_{\ell}$. In the following, we will prove statements at the level of torsion coefficients, and then pass to inverse limits to obtain statements at the level of $\bbZ_{\ell}$-coefficients (and, after tensoring with $\bbQ_{\ell}$, for $\bbQ_{\ell}$-coefficients). 
\end{remark}

\begin{remark}
If $\sigma: k \hookrightarrow \bbC$ is a fixed embedding, the we may also consider the singular cohomology $\rH^{i}(X_{\sigma}^{an}, \bbZ)$. The results of this section also hold in this setting. 
\end{remark}

\subsubsection{Connectivity over a point}
In this subsection, we shall work with schemes
$S$ of finite type over a separably closed field $k$,
and $\rH^{*}(S)$ will denote the etale cohomology groups as in the previous paragraph.
Let $X \subset \bbP^{n}$ be a closed subscheme and consider $\rJ^{[p]}(X) \subset \bbP^{(p+1)(n+1)-1}.$ \\

\begin{definition}
Let $X \subset \bbP^{n}$ be a closed subscheme and $d$ an integer such that $d \leq n$. Then $X$ is \emph{cohomologically $d$-connected}  if the restriction homomorphism
$$\rH^{i}(\bbP^{n}) \rightarrow \rH^{i}(X)$$ is an isomorphism for 
all $i <d$, and an injection for $i=d$. 
\end{definition}

\begin{remark}
We note that, if $ \Char(k) = 0$, standard results show that this notion will be independent of the prime $\ell$. In characteristic $p$, this would follow from Deligne's proof of the Weil conjectures if $X$ is also smooth. In general, it would follow from certain standard conjectures in algebraic geometry. For our purposes, we have simply fixed a prime $\ell$ not equal to the characteristic of $k$.
\end{remark}

We begin by proving the following connectivity property of the join. The analogous statement in the setting of singular cohomology was proven by the first author in (\cite{Basu-Toda}). Our goal here is to give a `motivic proof' of this statement which is applicable to any Weil cohomology theory. 

\begin{theorem}
\label{thm:connjoin}
Let $X \subset \bbP^{n}$ be a closed subscheme.
Then $\rJ^{[p]}(X) \subset \bbP^{(p+1)(n+1)-1}$ is cohomologically $p$-connected. In particular, the restriction 
homomorphism
$$\rH^{j}(\bbP^{(p+1)(n+1)-1}) \rightarrow \rH^{i}(\rJ^{[p]}(X))$$
is an isomorphism for $0 \leq j < p$, and an injection  for $j=p$.
\end{theorem}

We begin with some preliminary remarks. In the following, for any closed subscheme $X \subset \bbP^{n}$ we set $\rC'(X):= \rC(X) \setminus o_{X}$. Note that one has a natural cartesian diagram:
\begin{equation}
\label{eqn:punctcone}
\xymatrix{
\rC'(X) \ar[r] \ar[d] & \bbA^{n+1} \setminus 0 \ar[d] \\
X \ar[r] & \bbP^{n} }
\end{equation}
In particular, the natural projection $\rC'(X) \rightarrow X$ is a $\bbG_{m}$-bundle. \\

The following lemma is well-known. Over the complex numbers, it follows directly from the contractibility of the cone. We provide a proof here applicable to any `good cohomology theory' due to a lack of reference.

\begin{lemma}\label{lem:conecontract}
The natural inclusion $o_X \hookrightarrow C(X)$ induced an isomorphism on cohomology:
$$\rH^{i}(\rC(X)) \xrightarrow{\isom} \rH^{i}(o_X).$$
\end{lemma}
\begin{proof}
Let $Y$ denote the blow-up of $C(X)$ at $o_X$. Then it is a standard fact that there is a natural map $\pi: Y \rightarrow X$ which realizes $Y$ as a line bundle over $X$. Moreover, the exceptional fiber $E$ of the blow-up $Y$ is canonically identified with the zero section of $\pi$. In particular, $\rH^{\cdot}(Y) \isom \rH^{\cdot}(X)$ and $\rH^{\cdot}(E) \isom \rH^{\cdot}(X)$. 
On the other hand, one has the usual long exact sequence for the cohomology of the blow-up (\cite[\href{https://stacks.math.columbia.edu/tag/0EW5}{Tag 0EW5}]{SP}) :
$$
\cdots \rightarrow \rH^{i}(C(X)) \rightarrow \rH^{i}(o_X) \oplus \rH^{i}(Y) \rightarrow \rH^{i}(E) \rightarrow \rH^{i+1}(C(X)) \rightarrow \cdots,$$
where the arrows are induced by the natural pull-back maps on cohomology. Since the restriction homomorphism $\rH^{i}(Y) \rightarrow \rH^{i}(E)$ is an isomorphism (by the remarks above), the natural restriction homomorphisms $\rH^{i}(C(X)) \rightarrow \rH^{i}(o_X)$ must be isomorphisms.
\end{proof}

\begin{lemma}
\label{lem:punconcont}
With notation as above, one has 
$$\rH^{i}(\rC'(\rJ^{[p]}(X))) = 0 \text{ for all } 0 < i < p$$
and $$\rH^{0}(\rC'(\rJ^{[p]}(X))) = \rH^{0}(o).$$
\end{lemma}

Before proving the lemma, we give two proofs of Theorem~\ref{thm:connjoin}. 
The first uses a spectral sequence argument, while the second proof uses the following standard Gysin long exact sequence. 

\begin{lemma}\cite[Corollaire 1.5, Expos\'{e} VII]{SGA5}
\label{lem:gysin}
Let $Z$ be a scheme, and $X \rightarrow Z$ be a rank $r$ vector bundle. Let $U \subset X$ denote the complement of the zero section. Then there is a long exact sequence in cohomology:
$$\cdots \rightarrow \rH^{i -2r}(Z)(-r) \rightarrow   \rH^{i}(X) \rightarrow \rH^{i}(U) \rightarrow \cdots $$

\end{lemma}

We are now in a position to prove Theorem~\ref{thm:connjoin}.
\begin{proof}[Proof of Theorem~\ref{thm:connjoin}]
The conclusion follows by an application of Lemma~\ref{lem:punconcont} to the Leray spectral sequences for the $\bbG_{m}$ bundle $\pi: \rC'(\rJ^{[p]}(X)) \rightarrow \rJ^{[p]}(X)$. More precisely, the cartesian diagram \ref{eqn:punctcone} of $\bbG_m$-bundles gives rise to a commutative diagram of spectral sequences:
$$
\xymatrix{
\rE_2^{i,j}(X):=\rH^{i}(\rJ^{[p]}(X), R^{j}\pi_{*}(\bbZ/\ell^n\bbZ) \ar@2[r] & \rH^{i+j}(\rC'(\rJ^{[p]}(X))) \\
\rE_2^{i,j}(\bbP^{(p+1)(n+1)-1}):=\rH^{i}(\bbP^{(p+1)(n+1)-1}, R^{j}\pi_{*}(\bbZ/\ell^n\bbZ))) \ar@2[r] \ar[u] & \rH^{i+j}(\bbA^{(p+1)(n+1)} \setminus 0) \ar[u] }.
$$
Here, by abuse of notation, we use the same notation $\pi$ to denote the natural maps $\bbA^{(p+1)(n+1)} \setminus 0 \rightarrow \bbP^{(p+1)(n+1)-1}$ and 
$\rC'(\rJ^{[p]}(X)) \rightarrow \rJ^{[p]}(X)$. Since $\pi$ is a $\bbG_m$-bundle, $R^{j}\pi_{*}(\bbZ/\ell^n\bbZ)$ is a local system with stalk at $x \in \rJ^{[p]}(X)$ given by $\rH^{j}(\bbG_m)$, and similarly for $x \in \bbP^{(p+1)(n+1)-1}$. Moreover, in the case of $\bbP^{(p+1)(n+1)-1}$ it is the trivial local system. Since $R^{j}\pi_{*}(\bbZ/\ell^n\bbZ)$ on $\rJ^{[p]}(X)$ is the restriction of the corresponding local system on $\bbP^{(p+1)(n+1)-1}$ (due to the fact that the base change map is an isomorphism for zariski or etale local $G$-torsors as a consequence of the Kunneth formula, and the fact that \ref{eqn:punctcone} is cartesian), it is also a trivial local system. In particular, the cohomology groups $\rE_2^{i,j}(X)$ are zero for $j \neq 0,1$, and otherwise one has $\rE_2^{i,0}(X)=\rH^{i}(\rJ^{[p]}(X))$ and 
$\rE_2^{i,1}(X) = \rH^{i}(\rJ^{[p]}(X)) \otimes \rH^{1}(\bbG_m)$, and similarly for $\rE_2^{i,j}(\bbP^{(p+1)(n+1)-1})$. Note that one can identify
$\rH^{i}(\rJ^{[p]}(X)) \otimes \rH^{1}(\bbG_m) = \rH^{i}(\rJ^{[p]}(X))(-1)$.\\

It follows that both spectral sequences are concentrated in two columns and degenerate at $\rE_3$. In particular, they give rise to a commutative diagram of long exact sequences:

\begin{equation}
\label{eqn:commdiag}
\begin{tikzcd}[row sep=large, column sep=tiny]
\  \ar[r] & \rH^{i-1}(C'(\rJ^{[p]}(X))) \ar[r] & \rH^{i-2}(\rJ^{[p]}(X))(-1) \ar[r] & \rH^{i}(\rJ^{[p]}(X)) \ar[r] & \rH^i(C'(\rJ^{[p]}(X)))
\ar[r]
& \   \\ 
\  \ar[r]& \rH^{i-1}(\bbA^{N} \setminus 0) \ar[r] \ar[u]  & \rH^{i-2}(\bbP^{N-1})(-1) \ar[r]\ar[u] & \rH^{i}(\bbP^{N-1}) \ar[r] \ar[u] & \rH^i(\bbA^{N} \setminus 0) \ar[u] 
\ar[r]
& \ 
\end{tikzcd}
\end{equation}

where $N = (p+1)(n+1)$.
The result is now an easy consequence of Lemma~\ref{lem:punconcont}, and an application of the five lemma to the commutative diagram of long exact sequences above. 
\end{proof}
\begin{proof}[Alternate proof  of Theorem~\ref{thm:connjoin} using the Gysin]
Let $X' = \rJ^{[p]}(X)$ and $Y' \rightarrow X'$ be the line bundle in the proof of Lemma~\ref{lem:conecontract}. Similarly, let $X'' : = \bbP^{(p+1)(n+1)-1}$ and $Y'' \rightarrow X''$ the corresponding line bundle. Note that, in the case of $X''$, this is simply the tautological line bundle (i.e. the bundle given by the locally free sheaf $\cO(-1)$ on $X''$). Since 
$Y'$ is simply the restriction of $Y''$ to $X'$, it follows that $Y'$ is the line bundle associated to the locally free sheaf $\cO_{X'}(-1)$. We can now apply Lemma~\ref{lem:gysin} to both $Y' \rightarrow X'$ and $Y'' \rightarrow X''$ to get a commutative diagram of long exact sequences:

\[
\begin{tikzcd}[row sep=large, column sep=tiny]
\cdots \ar[r] & \rH^{i-2}( \rJ^{[p]}(X))(-1) \ar[r]  & \rH^{i}( \rJ^{[p]}(X))
\ar[r] & \rH^i(C'(\rJ^{[p]}(X))) \ar[r]&\cdots\\
\cdots \ar[r]& \rH^{i-2}(\bbP^{(p+1)(n+1)-1})(-1)\ar[u]  \ar[r]& \rH^{i}(\bbP^{(p+1)(n+1)-1}) \ar[r] \ar[u] & \rH^i(\bbA^{(p+1)(n+1)} \setminus 0) \ar[u] \ar[r]&\cdots
\end{tikzcd}
\]

Here we have identified $\rH^{*}(Y')$ with $\rH^{*}(X')$ (since this is a line bundle over $X'$), and similarly for $Y''$ and $\bbP^{(p+1)(n+1)-1}$.
This diagram is the same as diagram \ref{eqn:commdiag}, and one can proceed now as in the previous lemma.
\end{proof}

In the following, we shall make repeated use of the K\"{u}nneth formula for cohomology with coefficients in a principal ideal domain. We recall it here for the convenience of the reader. In particular, given schemes $X$ and $Y$ over $k$ (separably closed) one has the following  {\em K\"{u}nneth} short exact sequence for etale cohomology with $R = \bbZ/\ell^n\bbZ$-coefficients (\cite[\href{https://stacks.math.columbia.edu/tag/0F1P}{Tag 0F1P}]{SP}):
\footnote{Note that in loc. cit. it is shown that $R\Gamma(X,R) \otimes^{\bbL} R\Gamma(Y,R) \cong R\Gamma(X \times Y,R)$. This gives rise to the standard ${\rm Tor}$ spectral sequence. If $R =\bbZ/\ell^n\bbZ$, then all ${\rm Tor}^i$'s vanish for $i > 1$, and the spectral sequence gives the Kunneth short exact sequence.
}
\begin{equation}\label{K\"{u}nneth}
0 \rightarrow \bigoplus_{r+s =k} \rH^{r}(X) \otimes_R \rH^{s}(Y) \rightarrow \rH^{k}(X \times Y) \rightarrow \bigoplus_{r+s =k+1} {\rm Tor}_1^{R}(\rH^{r}(X),\rH^{s}(Y)) \rightarrow 0.
\end{equation}
If $R= \bbQ_{\ell}$, then the ${\rm Tor}$ term vanishes and one has the following simplified formula (see for example \cite[page 267]{Milne}):
$$ \bigoplus_{r+s =k} \rH^{r}(X) \otimes_R \rH^{s}(Y) \xrightarrow{\cong} \rH^{k}(X \times Y).$$

\begin{proof}[Proof of Lemma~\ref{lem:punconcont}]
We shall prove this by induction on $p$.  In the following, we denote by $o$ the cone point.\\
\begin{enumerate}[Step 1.]
\item
\label{itemlabel:lem:punconcont:1}
Suppose $p = 1$. In this case, we are reduced to showing that 
$\rC'(\rJ(X,X)) = (\rC(X) \times \rC(X)) \setminus o$ is connected. This is follows from Grothendieck's proof of Zariski's main theorem (or by hand). In fact, this is true more generally  for $\rC'(\rJ^{[r]}(X))$.\\
\item
\label{itemlabel:lem:punconcont:2}
Suppose $p= 2$. In this case, we are reduced to showing that
$\rH^{1}(\rC(X)^{\times 3} \setminus o) = 0$. Let $U := (\rC(X)^{\times 2} \setminus o) \times \rC(X)$ and $V := \rC(X)^{\times 2} \times \rC'(X)$. Then $\{U,V\}$ is an open cover of  $\rC(X)^{\times 3} \setminus o$ and the intersection
$U \cap V = (\rC(X)^{\times 2} \setminus o) \times \rC'(X).$
The Mayer-Vietoris sequence (and Step 1) gives an exact sequence:
$$
0 \rightarrow \rH^{1}(\rC(X)^{\times 3} \setminus o) \rightarrow
\rH^{1}(U) \oplus \rH^{1}(V) \rightarrow \rH^{1}(U \cap V) \rightarrow \cdots
$$
Note that the left most arrow is an injection, since the previous arrow in the Mayer-Vietoris sequence must be a surjection by Step 1.
By \ref{K\"{u}nneth} and \ref{lem:conecontract}, $\rH^{1}(U) = \rH^{1}(\rC(X)^{\times 2} \setminus o)$. Note that the cohomology of the point is zero except in degree $0$, where it is simply the coefficient ring $R$; in particular, the ${\rm Tor_1}$-terms in the K\"{u}nneth exact sequence vanish. Similarly, $\rH^{1}(V) = \rH^{1}(\rC'(X))$. Another application of the K\"{u}nneth exact sequence shows that the third arrow in the above sequence is an injection. It follows that  $\rH^{1}(\rC(X)^{\times 3} \setminus o) = 0$.\\
\item
\label{itemlabel:lem:punconcont:3}
Suppose the Proposition is known for all $m < p$. We need show that $\rH^{i}(\rC'(\rJ^{[p]}(X)) = 0$ for all $0 < i < p$. Let 
$U = \rC'(\rJ^{[p-1]}(X)) \times C(X)$ and $V = \rC(X)^{\times p} \times \rC'(X)$. Note that $\{U,V\}$ is an open cover of $\rC'(\rJ^{[p]}(X))$ and 
$U \cap V = \rC'(\rJ^{[p-1]}(X)) \times \rC'(X)$. By an application of the 
K\"{u}nneth exact sequence:
\begin{enumerate}[(a)]
\item
$\rH^{i}(U) = 0$ for all $0 < i < p-1$,
\item
$\rH^{i}(V) = \rH^{i}(\rC'(X))$ for all $i \geq 0$,
\item
$\rH^{i}(U \cap V) = \rH^{i}(\rC'(X)$ for all $i < p-1$.
\end{enumerate}
Therefore, an application of Mayer-Vietoris shows that $\rH^{i}(\rC'(\rJ^{[p]}(X))) = 0$ for all $0 < i < p-1$. Moreover, in degree $p-1$ one has an exact sequence:

$$0 \rightarrow \rH^{p-1}(\rC'(\rJ^{[p]}(X))) \rightarrow \rH^{p-1}(U) \oplus \rH^{p-1}(V) \rightarrow \rH^{p-1}(U \cap V) \rightarrow \cdots.$$
An argument via K\"{u}nneth, as in 
Step~\ref{itemlabel:lem:punconcont:2},
shows that the third arrow is injective and the result follows.
\end{enumerate}
\end{proof}

\begin{remark}
The result only uses formal properties of a cohomology theory (K\"{u}nneth, Mayer-Vietoris, Leray/Gysin) and contractibility of the cone. 
\end{remark}

Note that the proof of Theorem~\ref{thm:connjoin} holds verbatim in the multi-join setting of the following theorem.

\begin{theorem}\label{thm:connmultijoin}
Let for $0 \leq i \leq p$, $X_i \subset \bbP^{n_i}$  be  closed subschemes.
Then $\rJ^{[p]}(\bbX) \subset \bbP^{N}$ (with $N = \sum_{i=0}^{p}(n_i + 1) -1$) is cohomologically $p$-connected.
\end{theorem}

We shall now extend the connectivity result above to the relative setting. Suppose now that $S$ is a scheme of finite type over a field $k$ and $\cE$ is a vector bundle on $S$. Let $X$ be a closed subscheme of $\bbP(\cE)$. Then, as before, we have a natural embedding $\rJ^{[p]}_{S}(X) \hookrightarrow \bbP(\cE^ {\oplus (p+1)})$. Recall, we have a commutative diagram
\begin{equation} \label{eqn:basechange}
\xymatrix{
 \rJ^{[p]}_{S}(X) \ar[r] \ar[d]^{q_1}  & \bbP(\cE^ {\oplus (p+1)})_{\pi(X)} \ar[d]^{q_2} \\
\pi(X) \ar@{=}[r] & \pi(X).}
\end{equation}
where $\bbP(\cE^ {\oplus (p+1)})_{\pi(X)}$ is canonically isomorphic to $\bbP(\cF^ {\oplus (p+1)})$ with
$\cF := \cE|_{\pi(X)}$.\\

We have the following relative version of Theorem~\ref{thm:connjoin}. We state the proposition for etale cohomology with $\bbZ/\ell$ (or $\bbZ_{\ell}$ or $\bbQ_{\ell}$) coefficients with $\ell$ not equal to the characteristic of $k$. However, as will be clear from the proof, the same result holds in any good cohomology theory. The proof only uses the proper base change theorem and existence of a Leray spectral sequence.

\begin{theorem}\label{thm:relconnjoin}
With notation as above,
the natural map 
\[
\rH^{i}_{et}(\bbP(\cF^ {\oplus (p+1)})) \rightarrow \rH^{i}_{et}(\rJ^{[p]}_{S}(X))
\]
is an isomorphism
for $0 \leq i < p$, and an injection for $i=p$. 
\end{theorem}
\begin{proof}
The commutative diagram \eqref{eqn:basechange} gives rise to a morphism of sheaves
$$R^{i}q_{2,*}(\bbZ/\ell) \rightarrow R^{i}q_{1,*}(\bbZ/\ell).$$ 
Note that this map is an isomorphism on stalks for all $i < p$. To see this, we use the proper base change theorem to compute the stalks. In that case, one has isomorphisms: $$R^{i}q_{1,*}(\bbZ/\ell)_{s} \cong \rH^{i}_{et}(\rJ^{[p]}(X)_{s}) \cong \rH^{i}_{et}(\rJ^{[p]}(X_{s})).$$ The first is a consequence of proper base change, and the second follows from the base change property for joins. Similarly, we have isomorphisms:
$$R^{i}q_{2,*}(\bbZ/\ell) \cong \rH^{i}_{et}(\bbP(\cF^ {\oplus (p+1)})_{s}) \isom \rH^{i}_{et}(\bbP^{(p+1)(n+1) - 1})$$
where $n+1$ is the rank of $\cE$. An application of Theorem~\ref{thm:connjoin} now shows that the above higher direct images are isomorphisms for $j < p$. A Leray spectral sequence argument now gives the desired result.
\end{proof}

\begin{example}
\label{eg:poincare-product}
Suppose $S = \bbP^{m}$ and consider the trivial bundle $\cE$ of rank $n+1$ over $S$. Then $\bbP(\cE) = \bbP^{m} \times \bbP^{n}$. In that case, for $X \subset \bbP(\cE)$, the above result gives an isomorphism 
$$\rH^{j}_{et}(\pi(X) \times \bbP^{(p+1)(n+1) - 1}) \rightarrow \rH^{i}_{et}(\rJ^{[p]}_{S}(X)).$$ 
Here $\rJ^{[p]}_{S}(X) \subset \bbP^{m} \times \bbP^{(p+1)(n+1) - 1}$.
\end{example}

We conclude this section by noting that the proof of Theorem~\ref{thm:relconnjoin} also works in the relative multi-join setting. Let $\cE_i$ ($0 \leq i \leq p$) be vector bundles of rank $r_i$ on $S$. For each $i$, let $X_i$ be a closed subscheme of $\bbP(\cE_i)$. Then, as before, we have a natural embedding $\rJ_{S}(\bbX) \hookrightarrow \bbP(\bigoplus_i \cE_i) = \bbP(\bbE)$ 
(recall that we denote  by $\rJ_S(\bbX)$ the multiple join $\rJ_S(X_0,\cdots,X_p) $)
and a commutative diagram
\begin{equation}\label{eqn:basechange2}
\xymatrix{
 \rJ_{S}(\bbX) \ar[r] \ar[d]^{q_1}  & \bbP(\bbE)_{\pi(\bbX)} \ar[d]^{q_2} \\
\pi(\bbX) \ar@{=}[r] & \pi(\bbX).}
\end{equation}
where $\bbP(\bbE_{\pi(\bbX)}) $ is canonically isomorphic to $\bbP(\bbF)$ with
$\bbF := \bbE|_{\pi(\bbX)}$.\\

\begin{theorem}\label{thm:relconnjoin-multi}
With notation as above,
the natural map 
\[
\rH^{i}_{et}(\bbP(\bbF)) \rightarrow \rH^{i}_{et}(\rJ_{S}(\bbX))
\]
is an isomorphism
for $0 \leq j < p$, and an injection for $j=p$. 
\end{theorem}
\begin{proof}
We can argue as in the proof of the previous result, given Theorem~\ref{thm:connmultijoin}.

\end{proof}

\subsubsection{A generalization of the cohomological connectivity result}
In this section, we prove analogs of the results of the previous setting where a `higher' cohomological connectivity of the $X$ is assumed. We fix a separably closed base field $k$ as before. Moreover, we only consider etale cohomology with $\bbQ_\ell$-coefficients. \\

In this setting, we have the following analog of the Theorem~\ref{thm:connjoin}. 

\begin{theorem}\label{thm:iteratedhigherconn}
Let $X \subset \bbP^{n}$ be a cohomologically $d$-connected closed subscheme. Then $\rJ^{[p]}(X) \subset \bbP^{(p+1)(n+1)-1}$ is cohomologically $((p+1)d + p)$-connected. In particular, the restriction homomorphism
$$\rH^{i}(\bbP^{(p+1)(n+1)-1},\bbQ_{\ell}) \rightarrow \rH^{i}(\rJ^{[p]}(X),\bbQ_{\ell})$$
is an isomorphism for $0 \leq i < (p+1)d + p$, and an injection
for $i=(p+1)d+p$.
\end{theorem}
\begin{proof}
One can use the Gysin sequence, as in the second proof of Theorem~\ref{thm:connjoin}, given Lemma~\ref{lem:conehigherconn} below.

\end{proof}

\begin{remark}
 
\begin{enumerate}
\item The weak Lefschetz theorem states that any smooth complete intersection $X$ in $\bbP^{n}$ is cohomologically $(\dim(X) -1)$-connected.
\item The Barth-Larsen theorem \cite{Barth-Larsen} 
(and its generalization due to Ogus \cite{Ogus}, Hartshorne-Speiser \cite{Hartshorne-Speiser}) 
states that any local complete intersection projective variety $X \subset \bbP^{n}$ of dimension $r$ is $(2r - n)$-cohomologically connected.
\item We note that, even if $X$ is smooth, the iterated join will generally be not smooth. In particular, neither the weak Lefschetz nor the Barth-Larsen theorem apply in order to obtain cohomological connectivity results for the join.
\item On the other hand, we obtain many examples of $X$ satisfying 
the hypothesis of  Theorem~\ref{thm:iteratedhigherconn} by applying the previous remark in either the weak Lefschetz or Barth-Larsen settings.
\end{enumerate}
\end{remark}

\begin{lemma}\label{lem:conehigherconn}
Let $X \subset \bbP^{n}$ be a cohomologically $d$-connected closed subscheme. Then one has the following vanishing for the punctured cone:
$$\rH^{i}(C'(\rJ^{[p]}(X)),\bbQ_{\ell}) = 0 \text{ for all } 0 < i < (p+1)d + p.$$
If $i = 0$, then $\rH^{0}(C'(\rJ^{[p]}(X))) = \rH^{0}(o)$.

\end{lemma}
\begin{proof}
For simplicity, we drop the coefficients $\bbQ_\ell$ from the notation. One can argue as in the proof of Lemma~\ref{lem:punconcont}. We will show the main case of $p=1$, which follows from Lemma~\ref{lem:conehigher} below. The rest of the proof then proceeds exactly in the proof of Lemma~\ref{lem:punconcont}. So we suppose that $p=1$.\\

As before, we are interested in $\rC'(\rJ(X,X)) = (\rC(X) \times \rC(X)) \setminus o$. Let $U = \rC'(X) \times \rC(X)$ and $V = \rC(X) \times \rC'(X)$. Note that $U \cup V = \rC'(\rJ(X,X))$, and $U \cap V = 
\rC'(X)\times \rC'(X)$. By the K\"{u}nneth exact sequence, $\rH^{m}(U) = \rH^{m}(C'(X))$ and $\rH^{m}(V) = \rH^{m}(C'(X))$ for all $m$. In particular, both groups vanish for $0 < m < d$, and are given by the coefficients $R$ in degree $0$. 
The K\"{u}nneth formula applied to $U \cap V$ gives and isomorphism:
{\small
$$
 \bigoplus_{i+j =m} \rH^{i}(\rC'(X)) \otimes \rH^{j} \rC'(X)) \xrightarrow{\cong} \rH^{m}(U\cap V).
$$
}

If $m< 2d$, then by the previous remarks, the left term is equal to $\rH^{m}(U) \oplus \rH^{m}(V)$. In particular, an application of Mayer-Vietoris proves the desired result for $m < 2d$. In degree
$m = 2d$, one obtains a short exact sequence:
$$0 \rightarrow \rH^{2d}(U \cup V) \rightarrow \rH^{2d}(U) \oplus \rH^{2d}(V) \rightarrow \rH^{2d}(U\cap V).$$
The right arrow is injective by the previous remarks. This completes the proof in the case $p=1$. \\
We briefly discuss the case of $p > 1$. 
\label{itemlabel:lem:punconcont:3}
Suppose the result is known for all $m < p$. We need show that $\rH^{i}(\rC'(\rJ^{[p]}(X)) = 0$ for all $0 < i < (p+1)d+p$. Let 
$U = \rC'(\rJ^{[p-1]}(X)) \times C(X)$ and $V = \rC(X)^{\times p} \times \rC'(X)$. Note that $\{U,V\}$ is an open cover of $\rC'(\rJ^{[p]}(X))$ and 
$U \cap V = \rC'(\rJ^{[p-1]}(X)) \times \rC'(X)$. By an application of the
K\"{u}nneth formula:
\begin{enumerate}[(a)]
\item
$\rH^{i}(U) = 0$ for all $0 < i < dp + (p-1)$,
\item
$\rH^{i}(V) = \rH^{i}(\rC'(X))$ for all $i \geq 0$,
\item
$\rH^{i}(U \cap V) \cong \oplus_{r+s = i} \rH^{r}(U) \otimes \rH^{r}(\rC'(X))$ for all $i \geq 0$.
\end{enumerate}
Moreover, by Lemma \ref{lem:conehigher}, $\rH^i(V) = 0$ for all $0 < i < d$. As before, using these facts and Mayer-Vietoris gives the desired vanishing. 
\end{proof}

\begin{lemma}\label{lem:conehigher}
Let $X \subset \bbP^{n}$ be a cohomologically $d$-connected closed subscheme. Then one has the following vanishing for the punctured cone:
$$\rH^{i}(C'(X)) = 0 \text{ for all } 0 < i < d.$$
In degree $0$, $\HH^0(C'(X)) \cong R$ (where $R$ is the ring of coefficients).
\end{lemma}
\begin{proof}
Let $Y \rightarrow X$ be the line bundle as in the proof of Lemma~\ref{lem:conecontract}. We can now apply Lemma~\ref{lem:gysin}, and argue as in the `alternate' proof to Theorem~\ref{thm:connjoin} to get a commutative diagram of long exact sequences:
$$
\begin{tikzcd}[row sep=large, column sep=small]
\cdots \ar[r] & \rH^{i-2}(X)(-1) \ar[r]  & \rH^{i}(X)
\ar[r] & \rH^i(C'(X)) \ar[r] & \rH^{i-1}(X)(-1)\ar[r]&\cdots \\
\cdots \ar[r]& \rH^{i-2}(\bbP^{n})(-1)\ar[u]  \ar[r]& \rH^{i}(\bbP^{n}) \ar[r] \ar[u] & \rH^i(\bbA^{n+1} \setminus 0) \ar[u] \ar[r] & \rH^{i-1}(\bbP^{n})(-1) \ar[u]\ar[r]&\cdots
\end{tikzcd}
$$
The result now follows by induction and the five lemma.
\end{proof}

We note that the previous result can also be adapted to the setting of multi-joins and also the relative setting. Here we only state the result in the multi-join setting, and 
leave the 
proof to the reader.

\begin{theorem} \label{thm:highermultijoinconn}
Let for $0 \leq i \leq p$, $X_i \subset \bbP^{n_i}$ be closed cohomologically $d_i$-connected subschemes. 
Then $\rJ(\bbX) \subset \bbP^{N}$ is cohomologically $(d + p)$-connected, where 
$d = \sum_{i=0}^{p}d_i$ and $N = \sum_{i=0}^{p} (n_i+1) -1$.   
\end{theorem}

\subsection{Cohomological connectivity over non-algebraically closed fields}
\label{subsec:non-acf}
We discuss the case where $k$ is possibly a non-algebraically closed field. Let $\bar{k}$ denote a fixed separable closure of $k$, and $G$ denote the corresponding Galois group. For $X/k$, we denote by $X_{\bar{k}}$ its base change to $\bar{k}$. We fix a prime $\ell \neq  \Char(k)$, and let $\rH^{i}(X)$ denote the etale cohomology with $\bbQ_{\ell}$-coefficients. Note that there is a natural continuous action of $G$ on $\rH^{i}(X_{\bar{k}})$.\\

The results of the previous sections give the following natural connectivity of the join with Galois action.

\begin{corollary}
\label{cor:relconnjoin}
Let $X_i \subset \bbP_k^{n_i}$ ($0 \leq i \leq p$) be closed cohomologically $d_i$-connected subschemes. Then $\rJ(\bbX)_{\bar{k}} \subset 
\bbP^{N}_{\bar{k}}$ is cohomologically $(d+p)$-connected, where $d = \sum_{i=0}^{p}d_i$ and $N = (\sum_{i=0}^{p}n_i+1) -1$. In particular,
\begin{equation}
\rH^{j}(\bbP^{N}_{\bar{k}}) \rightarrow \rH^{i}(\rJ(\bbX)_{\bar{k}})
\end{equation}
is an isomorphism of Galois modules for $0 \leq j <d+p$,  and injective for $j =d+p$.
\end{corollary}
\begin{proof}
This is a direct consequence of the functoriality of the restriction map, and the fact that the join construction is compatible with base extension. More precisely, $\rJ(\bbX_{\bar{k}}) = \rJ(\bbX)_{\bar{k}}$.
\end{proof}

In this setting, one has the usual Hochschild-Serre spectral sequence:
$$
\rE_2^{i,j} : = \rH^{i}(G,\rH^{j}(X_{\bar{k}})) \Rightarrow \rH^{i+j}(X)
$$
where $\rH^{i}(G,\rH^{j}(X_{\bar{k}}))$ is the Galois cohomology of $G = \Gal(\bar{k}/k)$ with coefficients in the Galois module $\rH^{j}(X_{\bar{k}})$.

\begin{corollary}
With notation and assumptions as in the previous corollary,
the subscheme $\rJ(\bbX) \subset \bbP^{N}$ 
is cohomologically $(d+p)$-connected.
\end{corollary}
\begin{proof}
One has a commutative diagram of spectral sequences:
$$
\xymatrix{
 \rE_2^{i,j}(\bbX) : = \rH^{i}(G,\rH^{j}(\rJ(\bbX)_{\bar{k}}))\ar@2[r] & \rH^{i+j}(\rJ(\bbX)) \\
\rE_2^{i,j}(\bbP^{N}):=\rH^{i}(G,\rH^j(\bbP^{N}_{\bar{k}})) \ar@2[r] \ar[u] & \rH^{i+j}(\bbP^{N}) \ar[u] }.
$$
By the previous corollary, the $\rE_{2}^{i,j}$-terms are isomorphic for $0 \leq j <d+p$, and therefore also on the corresponding $\rE_{\infty}$ terms.
\end{proof} 

\subsection{Cohomological connectivity and Poincar\'e polynomials}
\label{subsec:poincare}

We now prove a key inequality relating the Poincar\'e polynomial of a closed subscheme $X \subset \bbP^m \times \bbP^n$, with that of $\pi(X)$ where $\pi: \bbP^m \times \bbP^n \rightarrow \bbP^m$ is the projection morphism.

\begin{definition}[Poincar\'e Polynomial]\label{defn:poincarepoly}
\label{def:poincare}
For any scheme of finite type $X$ over a field $k$,
we will denote
\[
P(X):=\sum_{i} 
b_{i}(X)T^{i} \in \bbZ[T],
\]
where $b_{i}(X) : = \dim_{\bbQ_\ell}(\rH^{i}(X, \bbQ_\ell))$,
and $\ell$ is a prime not equal to the characteristic of the field $k$.
\end{definition}

We have the following direct consequence of Theorem~\ref{thm:relconnjoin}.

\begin{theorem}
\label{thm:poincare}
With notation as in Theorem~\ref{thm:relconnjoin}, let $S =  \bbP^{m}$, 
$X \subset \bbP^{n} \times \bbP^{m}$, and $\pi: \bbP^{n} \times \bbP^{m} \to \bbP^{m}$ the projection morphism.
Then,
\begin{eqnarray*}
P({[\rJ^{[p]}_{S}(X)}) &\equiv&  P({\pi(X)}) (1+T^{2}+ T^{4} +\cdots + T^{2((p+1)(n+1)-1)}) \  {\rm mod}  \ T^{p} .
\end{eqnarray*}
\end{theorem}

\begin{proof}
Direct consequence of Theorem~\ref{thm:relconnjoin}.
\end{proof}

\section{Quantifier elimination, cohomology and joins}
\label{sec:qe-join}
In this section, we state and prove our result on cohomological quantifier elimination.
Let $k$ be a fixed algebraically closed field. 
We consider etale cohomology with coefficients 
in $\bbQ_\ell$ with $\ell \neq \Char(k)$.

\begin{notation}  
For any finite tuple $\bn=(n_1,\ldots,n_{m}) \in \bbN^{m}$, we denote:
\begin{enumerate}
\item  
$|\bn| = \sum_i n_i$;
\item
$\bbP^\bn= \bbP^{n_1} \times \cdots \times \bbP^{n_m}$.
\end{enumerate}
\end{notation}

In the following we will denote by bold letters $\bW^{(i,j,\ldots)},\bX^{(i,j,\ldots)}$ tuples of variables and we will denote by 
$|\bW^{(i,j,\ldots)}|, |\bX^{(i,j,\ldots)}|$ the lengths  of the corresponding tuples.

\begin{definition}[Proper formulas]
\label{def:proper-formulas}
Let  $\phi(\bX^{(1)};\ldots;\bX^{(n)})$   (with each $\bX^{(i)}$ denoting a  tuple of variables 
$(X_{i,0},\ldots,X_{i,n_i})$) be a quantifier-free first order formula 
in the language of fields with parameters in $k$.
We say that $\phi$ is
a \emph{quantifier-free proper formula}  (with $n$ homogeneous blocks)  if its atoms are of the form 
$P = 0$, where $P \in k[\bX^{(1)};\cdots;\bX^{(n)}]$ is a multi-homogeneous polynomial,
and $\phi$ does not contain any negations. \\  

We say that a first order formula in the language of fields with parameters in $k$ (possibly with quantifiers)
$$
\displaylines{
\phi(\bW^{(1)};\cdots;\bW^{(m)}) := (\bQ_0 \bX^{(1)}) \cdots (\bQ_n \bX^{(n)})\psi(\bW^{(1)};\cdots;\bW^{(m)};\bX^{(1)};\ldots;\bX^{(n)}),\cr
 \bQ_i \in \{\exists,\forall\}, 1 \leq i \leq n,
}
$$
is  a \emph{proper formula} (with $m$ homogeneous blocks), if $\psi$ is a quantifier-free proper formula.\\

A proper formula
$$
\displaylines{
\phi(\bW^{(1)};\cdots;\bW^{(m)}) := (\bQ_0 \bX^{(1)}) \cdots (\bQ_n \bX^{(n)})\psi(\bW^{(1)};\cdots;\bW^{(m)};\bX^{(1)};\ldots;\bX^{(n)}),\cr
 \bQ_i \in \{\exists,\forall\}, 1 \leq i \leq n,
}
$$
defines an algebraic subset of 
$\bbP^\mathbf{m}$, where $\mathbf{m}  = (|\bw^{(1)}| - 1, \ldots, |\bw^{(m)}| - 1)$ 
whose $k$-points are described by 
$$\displaylines{
(\bQ_1\bx^{(1)} \in \bbP^{|\bx^{(1)}|-1}(k)) \cdots (\bQ_n \bx^{(n)} \in \bbP^{|\bx^{(n)}|-1}(k))\psi(\bw^{1)};\cdots;\bw^{(m)};\bx^{(1)};\ldots;\bx^{(n)}),
}
$$
We denote this algebraic set by $\Reali(\phi)$ (the realization of $\phi$).
\end{definition}

\begin{notation}
\label{not:odd-even}
Given $P  = \sum_{i \geq 0} a_i T^i \in \bbZ[T]$, we write
\[
P \defeq P^{{\mathrm{even}}}(T^2) + T P^{{\mathrm{odd}}}(T^2),
\]
where 
\[
P^{{\mathrm{even}}} = \sum_{i \geq 0} a_{2i} T^i,
\]
and
\[
P^{{\mathrm{odd}}}= \sum_{i \geq 0} a_{2i+1} T^i.
\]
\end{notation}

Following \cite{Basu-Toda},
we introduce for any subscheme
$V \subset \bbP^{\bn}$, 
a  polynomial, $Q(V) \in \bbZ[T]$, which we call
the \emph{pseudo-Poincar\'e polynomial} of $V$
defined as follows.

\begin{equation*}
Q(V)\; \defeq \; \sum_{j \geq 0} (b_{2j}(V) - b_{2j-1}(V))T^j.
\end{equation*}

In other words,

\begin{equation}
\label{eqn:pseudo-Poincarepolynomial2}
Q(V)=  P(V)^{{\mathrm{even}}} - T P(V)^{{\mathrm{odd}}}. 
\end{equation}
For any proper formula $\phi$, we will denote:

\begin{equation}
\label{eqn:pseudo-Poincarepolynomial3}
Q(\phi) = Q({\Reali(\phi)}).
\end{equation}

Note that for each $n \geq 0$, 
\begin{equation}
\label{eqn:pseudo-poincare-of-projective}
Q({\bbP^n}) = 1+T +\cdots + T^n.
\end{equation}

We introduce below notation for several operators on polynomials that 
we will use later.

\begin{notation}[Operators on polynomials]
\label{not:poloperators}
\begin{enumerate}

\item
For any finite tuple $\bn$ of natural numbers, 
we denote by $\Rec_\bn: \bbZ[T]_{\leq 2|\bn|} \rightarrow \bbZ[T]_{\leq 2|\bn|}$,  the map defined by 
\[
\Rec_\bn(Q) = Q({\bbP^{\bn}}) - T^{2 |\bn|} Q(1/T).
\]
\item
For $0 \leq m \leq n$, 
we denote by $\Trunc_{m,n}: \bbZ[T]_{\leq n} \rightarrow \bbZ[T]_{\leq m}$
and $Q \in \bbZ[T]_{\leq n}$, we denote the map defined by:
for $Q = \sum_{i=0}^{n} a_i T^i \in \bbZ[T]_{\leq n}$,
\[
\Trunc_{m,n}(Q) = \sum_{0 \leq i \leq m} a_i T^i.
\]
\end{enumerate}
\end{notation}

Now let
$\psi(\bW^{1)};\cdots;\bW^{(m)};\bX^{(1)};\ldots;\bX^{(n)})$ be a quantifier-free proper formula with $m+n$ homogeneous
blocks. For $1 \leq i \leq m$, let $e_i = |\bW^{(i)}|-1$, and  $1 \leq j \leq n$, let $f_j = |\bX^{(j)}|-1$, and define $N_i,d_i,m_i$ by the formulas:
\begin{eqnarray*}
d_0 &=& \sum_{i=1}^{m} e_i, \\
N_1 &=& 1, \\
d_1&=& d_0 +  N_1(2(d_{0}+1)(f_{1} +1)-1),\\
m_1 &=& 2(d_0+1)(f_1+1) -1),
\end{eqnarray*}
and for $2 \leq j \leq m$, 
\begin{eqnarray*}
N_j &=& 2 N_{j-1}(d_{j-2} +1), \\
d_j &=& d_{j-1} + N_j (2(d_{j-1}+1)(f_{j} +1)-1),\\
m_j &=& 2(d_{j-1}+1)(f_{j} +1)-1.
\end{eqnarray*}

\begin{notation}
\label{not:generalized-join}
We will denote by $J_{m,n}(\psi)$ the \emph{quantifier-free}  proper formula (with 
$m + \sum_{j=1}^n N_j$ homogeneous blocks)
defined by
\begin{equation}
J_{m,n}(\psi) :=  
\bigwedge_{i_1 = 0}^{2 d_0+1} \cdots \bigwedge_{i_n = 0}^{2 d_{n-1}+1}\psi(\bW^{(1)};\cdots;\bW^{(m)};\bX^{(i_1)};\ldots;\bX^{(i_1,\ldots,i_n)}),
\end{equation}
where for each tuple $(i_1,\ldots,i_{j-1}) \in [0,2d_0+1] \times \cdots \times [0,2d_{j-2}+1]$,
$|X^{(i_1,\ldots,i_{j-1},0)}| = \cdots = |X^{(i_1,\ldots,i_{j-1}, 2d_{j-1}+1)}| = f_{j}$, and the tuples
$(X^{(i_1,\ldots,i_{j-1},0)}: \cdots : X^{(i_1,\ldots,i_{j-1}, 2d_{j-1}+1)})$ represent  homogeneous
coordinates in $\bbP^{m_j}$. If $V = \Reali(\psi)$, then we will denote by $J_{m,n}(V) = \Reali(J_{m,n}(\psi))$.
\end{notation}

\begin{remark}
\label{rem:generalized-join}
Notice that the realization, $\Reali(J_{m,n}(\psi))$,  is an algebraic subset of 
\[
\bbP^{e_1}  \times \cdots \times \bbP^{e_m}\times \bbP^{m_1} \times \cdots \times \underbrace{\bbP^{m_i} \times \cdots  \times \bbP^{m_i}}_{N_i} \times \cdots  \times 
\underbrace{\bbP^{m_n} \times \cdots \times \bbP^{m_n}}_{N_n}.
\]

Also notice that for each $j, 2 \leq j \leq n$,
$N_j = \prod_{h=2}^{j} (2(d_{h-2} +1))$ and we will index the factors of the product  
$\underbrace{\bbP^{m_i} \times \cdots  \times \bbP^{m_i}}_{N_i}$ by tuples
$(i_1,\ldots, i_{j-1}) \in [0, 2d_0+1] \times \cdots \times [0,2d_{j-2}+1]$.
\end{remark}

For each $i, 1 \leq i \leq n$,  let
\[
\mathbf{m}_i =(e_1,\ldots,e_m,{m_1}, \underbrace{{m_2}, \ldots , {m_2}}_{N_2},\ldots, \underbrace{{m_i}, \ldots , {m_i}}_{N_i}).
\]

For $\omega \in \{\exists,\forall\}^{[1,n]}$, we denote 
$$
\displaylines{
\psi^\omega(\bW^{(1)};\cdots;\bW^{(m)}) := (\omega(1) \bX^{(1)}) \cdots (\omega(n)  \bX^{(n)})\psi(\bW^{1)};\cdots;\bW^{(m)};\bX^{(1)};\ldots;\bX^{(n)}),
}
$$
and 
for $1 \leq i \leq n$,
\begin{eqnarray*}
F_i^\omega &=& \Trunc_{d_i,d_{i+1}+N_{i+1}} \circ (1-T)^{N_{i+1}}, \mbox{ if $\omega(i) = \exists$}, \\
&=&   \Rec_{\mathbf{m}_i} \circ \Trunc_{d_i, d_{i+1}+N_{i+1}} \circ (1-T)^{N_{i+1}} \circ \Rec_{\mathbf{m}_{i+1}}, \mbox{ if $\omega(i) = \forall$}.
\end{eqnarray*}

We denote:
\begin{eqnarray}
\label{eqn:F-omega}
F^\omega = F_1^\omega  \circ F_2^\omega  \circ \cdots \circ F_n^\omega.
\end{eqnarray}

With the above notation we have the following theorem which relates the pseudo-Poincar\'e polynomial of a quantified
proper formula, $\psi^\omega$, with that of the \emph{quantifier-free} proper formula $J_{m,n}(\psi)$.

\begin{theorem}
\label{thm:qe}
For each $\omega \in \{\exists,\forall\}^{[1,n]}$,
\[
Q({\psi^\omega}) =  F^\omega (Q({J_{m,n}(\psi)})).
\]
\end{theorem}
(Notice that in the statement of Theorem~\ref{thm:qe} the 
quantifier-free formula $J_{m,n}(\psi)$ does not depend on the sequence of quantifiers $\omega$, and only the operator
$F^\omega$ depends on $\omega$.)

The following special case of Theorem~\ref{thm:qe} will be important in the application of Theorem~\ref{thm:qe} in the proof
of an algebraic version of Toda's theorem. With the same notation as in Theorem~\ref{thm:qe}, suppose additionally that
$m =0 $. In this case, the formula $\psi$ has no free variables and is a sentence, 
and we have:

\begin{corollary}
\label{cor:qe}
\[
\psi   \Leftrightarrow (F^\omega(Q({J_{0,n}(\psi)})) = 1).
\]
\end{corollary}

\begin{proof}
Follows immediately from Theorem~\ref{thm:qe}.
\end{proof}

\subsection{An example}
\label{subsec:example}
Before we prove Theorem~\ref{thm:qe} it is instructive to consider an example.

\begin{example}
\label{eg:qe}
Let $m=1, n=2$, $e_1=f_1 =f_2 = 1$, and consider the quantifier-free proper formula:
$$
\displaylines{
\psi(\bW^{(1)};\bX^{(1)};\bX^{(2)}) := \cr
 ((W_{1,0} - W_{1,1}= 0) \wedge  (X_{1,0} - X_{1,1} = 0)) \cr \bigvee \cr
((W_{1,0} - 2 W_{1,1}= 0) \wedge  (X_{1,0} - 2 X_{1,1} = 0) \wedge (X_{2,0}- 2 X_{2,1} = 0)).
}
$$

The values of the various $N_i,d_i, m_i \mathbf{m}_i$ are displayed in the following table.

\begin{table}[h!]
\begin{center}
\begin{tabular}{l|l|l|l|l}
$i$ &$N_i$ & $d_i$ &$m_i$&$\mathbf{m}_i$\\
\hline
$0$&-&$1$&-&-\\
$1$&$1$&$8$ &$7$ &$(1,7)$\\
$2$ &$4$ &$148$ &$35$ &$(1,7,35^4)$\\
\end{tabular}
\end{center}
\end{table}

It is easy to check that $\Reali(J_{1,2}(\psi))$ is an algebraic subset of 
$\bbP^1 \times \bbP^7 \times \bbP^{35} \times \bbP^{35} \times \bbP^{35} \times \bbP^{35}$,
and 
\begin{eqnarray}
\nonumber
Q({J_{1,2}(\psi)})& =&  Q({\bbP^3 \times \bbP^{35} \times \bbP^{35}  \times \bbP^{35}  \times \bbP^{35}}) +
Q({ \bbP^3 \times \bbP^{17}  \times \bbP^{17}  \times \bbP^{17}  \times \bbP^{17}}) \\
&=& \frac{(1 - T^4)(1 - T^{36})^4}{(1-T)^5} + \frac{(1 - T^4)(1 - T^{18})^4}{(1-T)^5}.
\end{eqnarray}

Let $\omega,\omega' \in \{\exists,\forall\}^{[1,2]}$ be defined by
$$\displaylines{
\omega(1) = \exists, \omega(2) = \forall, \cr
\omega'(1) = \forall, \omega'(2) = \exists.
}
$$

It is easy to check that
\begin{eqnarray*}
Q({\psi^\omega}) &=& 1, \\
Q({\psi^{\omega'}}) &=& 0.
\end{eqnarray*}

Moreover, using Eqn. \eqref{eqn:F-omega} we have that:

\begin{eqnarray*}
F^\omega_1 &=& \Trunc_{1,9} \circ (1-T),\\
F^\omega_2 &=& \Rec_{(1,7)}\circ \Trunc_{8,152} \circ (1-T)^4\circ \Rec_{(1,7, 35^4)}, \\
F^{\omega'}_1 &=& \Rec_{(1)}\circ \Trunc_{1,9} \circ (1-T) \circ \Rec_{(1,7)}, \\
F^{\omega'}_2&=&  \Trunc_{8,152} \circ (1-T)^4.
\end{eqnarray*}

A calculation using the package Maple  now yields:
\begin{eqnarray*}
F^\omega(Q({J_{1,2}(\psi)})) &= & 1, \\
F^{\omega'}(Q({J_{1,2}(\psi)})) &=& 0.
\end{eqnarray*}
\end{example}

\subsection{Proof of 
the cohomological quantifier elimination theorem
}

Before we prove Theorem~\ref{thm:qe} we need a few preliminary facts.

\begin{theorem}[Alexander duality]
\label{the:alexanderduality_general}
Let $ V \subset \bbP^{\bn}$ be a closed subscheme.
Then for each odd $i$, $1\leq i \leq |\bn|$: 
\begin{equation}
\label{eqn:alexanderduality_general}
b_{i-1}(V) - b_{i-2}(V) = b_{2|\bn| -i}(\bbP^{\bn} \setminus  V) - 
b_{2|\bn| -i+1 }(\bbP^{\bn} \setminus  V) + 
b_{i-1}(\bbP^{\bn}).  
\end{equation}
\end{theorem}

\begin{proof}
Let $X = \bbP^\bn$ and  $U = X \setminus V$. Then, there is a long exact sequence
\[
\cdots \rightarrow \HH_V^p(X) \rightarrow \HH^p(X) \rightarrow \HH^p(U)\rightarrow \cdots
\]
and Alexander duality gives, 
\[
\HH_V^p(X) \cong \chk{\HH^{2|\bn| -p}(V)}.
\]

Eqn. \eqref{eqn:alexanderduality_general} now follows f that 
$\HH^p(X) = \HH^p(\bbP^\bn) = 0$ for all odd $p$.
\end{proof}

\begin{corollary}
\label{cor:alexanderduality_general}
Let $V \subset \bbP^{\bn}$ 
be 
a
closed subscheme.

Then,
$$
\displaylines{
Q(V) = Q({\bbP^{\bn}})
-  \mathrm{Rec}_{|\bn|}(Q({\bbP^{\bn} \setminus V})).
}
$$
\end{corollary}

\begin{theorem}
\label{the:compactcovering_general}
Let $\bn=(n_1,\ldots,n_m)$, $V \subset \bbP^\bn$  a closed subscheme.
Let $W = \bbP^\bn \setminus V$.
For each  $p \geq 0$,  and $0 \leq i < p$, we have that
\begin{enumerate}[1.]
\item
\label{itemlabel:thm:cc:1}
\[
\HH^i(\bbP^{(n_1+1)(p+1)-1} \times \pi_{\bn,1}(V))  \rightarrow \HH^i(\rJ^p_{\pi_{\bn,1}}(V))
\]
and
\item
\label{itemlabel:thm:cc:2}
\[
\HH^i(\bbP^{(n_1+1)(p+1)-1} \times \pi_{\bn,1}(W))  \rightarrow \HH^i(\bbP^\bN \setminus \rJ^{[p]}_{\pi_{\bn,1}}(V))
\]
are isomorphisms.
\end{enumerate}
\end{theorem}

\begin{proof}
The proof of 
Part \eqref{itemlabel:thm:cc:1}
follows from the argument in  Example~\ref{eg:poincare-product} with
$S$ replaced by $\bbP^{\bn'}$, and omitted. 
We now prove 
Part \eqref{itemlabel:thm:cc:2}.
Let  $U = \pi_{\bn,1}(W)$ and let 
\[
Z = \rJ^{[p]}_{\pi_{\bn,1}}(V) \cap (\bbP^{(n_1+1)(p+1)-1} \times U).
\]
There is a long exact sequence
$$\displaylines{
\cdots \rightarrow \HH_Z^i(\bbP^{(n_1+1)(p+1)-1} \times U) \rightarrow  
\HH^i(\bbP^{(n_1+1)(p+1)-1} \times U) \rightarrow \cr
\HH^i(W) \rightarrow  \HH_Z^{i+1}(\bbP^{(n_1+1)(p+1)-1} \times U) \rightarrow \cdots
}
$$
Using Alexander duality one has
\[
\HH_Z^i(\bbP^{(n_1+1)(p+1)-1} \times U) \cong \HH^{2((n_1+1)(p+1)-1 + |\bn'|)-i}(Z).
\]
Moreover,
\[
\dim Z \leq (n_1+1)(p+1)-1+ |\bn'| -(p+1),
\]
which implies that 
 \[
\HH_Z^{i+1}(\bbP^{(n_1+1)(p+1)-1} \times U) \cong \HH^{2((n_1+1)(p+1)-1 + |\bn'|)-i-1}(Z) = 0
\]
whenever
\[
2((n_1+1)(p+1)-1 + |\bn'|)-i-1> 2( (n_0+1)(p+1)-1+ |\bn'| -(p+1)) \Leftrightarrow i < p,
\] 
and in this case
\[
\HH_Z^{i}(\bbP^{(n_1+1)(p+1)-1} \times U) = 0
\]
as well.
\end{proof}

With the same notation as in Theorem~\ref{the:compactcovering_general}:
\begin{corollary}
\label{cor:compactcovering_general}
Let $p = 2m +1$ with $m \geq 0$. Then
\begin{eqnarray}
\label{eqn:complexjoin2.1'}
Q({\pi_{\bn,1}(V)})  &=&  (1 - T)\;Q({\rJ^{[p]}_{\pi_{\bn,1}}(V)})  \mod T^{m+1}, \\
\label{eqn:complexjoin2.2'}
Q({\pi_{\bn,1}(W)} ) &=&  (1 - T) Q({\bbP^\bN - \rJ^{[p]}_{\pi_{\bn,1}}(V)}) \mod T^{m+1}.
\end{eqnarray}
\end{corollary}

We will also need the following lemma.

\begin{lemma}
\label{lem:product}
Let $ p \geq 0$, and for each $1 \leq i \leq n$, let $V_i \subset \bbP^{n}$  be a closed subscheme, and $W_i = \bbP^n \setminus V_i$.
For $1 \leq i \leq n$, let $\pi_i: \bbP^{n} \times \cdots \times \bbP^{n} \rightarrow \bbP^{n}$ denote
the canonical surjection to the $i$-th factor.
\begin{enumerate}[1.]
\item
\label{itemlabel:lem:product:1}
Suppose that the restriction homomorphism $\HH^j(\bbP^{n}) \rightarrow \HH^j(V_i)$ is an isomorphism for 
$0 \leq j \leq p$. Then, the restriction homomorphism
\[
\HH^j(\bbP^{n} \times \cdots \times \bbP^{n}) \rightarrow \HH^j(\bigcap_{i=1}^n \pi_i^{-1}(V_i))
\] 
is an isomorphism for $0 \leq j \leq p$.
\item
\label{itemlabel:lem:product:2}
Suppose that the restriction homomorphism $\HH^j(\bbP^{n}) \rightarrow \HH^j(W_i)$ is an isomorphism for 
$0 \leq j \leq p$. Then, the restriction homomorphism
\[
\HH^j(\bbP^{n} \times \cdots \times \bbP^{n}) \rightarrow \HH^j(\bigcup_{i=1}^n \pi_i^{-1}(W_i))
\] 
is an isomorphism for $0 \leq j \leq p$.
\end{enumerate}
\end{lemma}

\begin{proof}
Easy.
\end{proof}

\begin{proof}[Proof of Theorem~\ref{thm:qe}]
For $0 \leq j \leq n$, let
$\phi^\omega_j(\bW^{(1)};\cdots;\bW^{(m)};\bX^{(i_1)};\cdots;\bX^{(i_1,\ldots,i_j)})$ denote the 
formula
\[ (\omega(j+1) \bX^{(i_1,\ldots,i_{j+1})}) \cdots  (\omega(n) \bX^{(i_1,\ldots,i_{n})})\psi(\bW^{(1)};\cdots;\bW^{(m)};\bX^{(i_1)};\cdots;\bX^{(i_1,\ldots,i_n)}),
\]
and let
$\psi^\omega_j$ denote  the formula
\[
\bigwedge_{i_1=0}^{2d_0+1} \cdots \bigwedge_{i_j =0}^{2d_{j-1}+1} \phi^\omega_j(\bW^{(1)};\cdots;\bW^{(m)};\bX^{(i_1)};\cdots;\bX^{(i_1,\ldots,i_j)}).
\] 

Notice that
\begin{eqnarray}
\label{eqn:proof:qe:1}
\psi^\omega_0 &=& \psi^\omega, \\
\label{eqn:proof:qe:2}
\psi^\omega_n  &=& J_{m,n}(\psi).
\end{eqnarray}

We prove by induction on $j$ that
\begin{eqnarray}
\label{eqn:proof:qe:3}
Q({\psi^\omega}) &=& F_1^\omega\circ \cdots \circ F_j^\omega (Q({\psi^\omega_j})).
\end{eqnarray}

Notice that \eqref{eqn:proof:qe:3} is true for $j=0$ using \eqref{eqn:proof:qe:1}, and implies the theorem in 
the case $j=n$ using \eqref{eqn:proof:qe:2}.

Now assume that \eqref{eqn:proof:qe:3} holds for $j \geq 0$ and we prove it for $j+1$, thus completing the  inductive step.

There are two cases to consider.

\begin{enumerate}[{Case} 1.]
\item
$\omega(j+1) = \exists $.
For each 
$(\bar{\bw};\bar{\bx}) \in \bbP^{\mathbf{m}_j}$
(where 
$\bar{\bw} = (\bw^{(1)};\cdots\; \bw^{(m)} \in \bbP^{e_1}\times \cdots \times \bbP^{e_m}$,
$\bar{\bx} = (\bar{\bx}_1;\cdots;\bar{\bx}_{j})$,
and for $1 \leq h \leq j$, 
$\bar{\bx}_h = (\cdots; \bx^{(i_1,\ldots,i_{h-1})};\cdots) \in  \underbrace{\bbP^{m_h} \times \cdots  \times \bbP^{m_h}}_{N_h}$),
and each tuple $(i_1,\ldots,i_{j}) \in [0,2d_0+1] \times \cdots \times [0, 2d_{j-1}]$, 
let 
$V^{(i_1,\ldots,i_j)}_{\bar{\bw};\bar{\bx}}$
denote the algebraic set 
\[
\Reali(\phi^\omega_{j+1}(\bw^{(1)};\cdots;\bw^{(m)}; \bx^{(i_1)};\bx^{(i_1,i_2)};\cdots;\bx^{(i_1,\ldots,i_{j-1})}; \bX^{(i_1,\ldots,i_{j})})) \subset \bbP^{m_{j+1}}.
\]
Notice that, for 
$0 \leq i \leq 2d_{j}$,
the restriction homomorphism
\[
\HH^i(\bbP^{m_{j+1}}) \rightarrow \HH^i(V^{(i_1,\ldots,i_j)}_{\bar{\bw};\bar{\bx}})
\]
is an isomorphism
using Part~\eqref{itemlabel:thm:cc:1} of Theorem~\ref{the:compactcovering_general}.

Also, observe that denoting by
\[
\pi_{(i_1,\ldots,i_j)}: \underbrace{\bbP^{m_{j+1}} \times \cdots \times \bbP^{m_{j+1}}}_{N_{j+1}} \rightarrow \bbP^{m_{j+1}},
\]
the projection on the $(i_1,\ldots,i_j)$-th factor, 
\[
\Reali(\psi^\omega_{j+1}(\bar{\bw};\bar{\bx};\cdot)) = \bigcap_{(i_1,\ldots,i_{j}) \in [0,2d_0+1] \times \cdots \times [0, 2d_{j-1}]} \pi_{(i_1,\ldots,i_j)}^{-1}(V^{(i_1,\ldots,i_j)}_{\bar{\bw};\bar{\bx}} ).
\]

Now using Part \eqref{itemlabel:lem:product:1} of Lemma~\ref{lem:product} we get that
for each point $(\bar{\bw};\bar{\bx}) \in  \Reali(\psi^\omega_j) \subset \bbP^{\mathbf{m}_{j}}$,
and for $0 \leq i \leq 2d_{j}$ the restriction homomorphisms
\[
\HH^i(\underbrace{\bbP^{m_{j+1}} \times \cdots \times \bbP^{m_{j+1}}}_{N_{j+1}}) \rightarrow 
\HH^i(\Reali(\psi^\omega_{j+1}(\bar{\bw};\bar{\bx};\cdot)))
\]
are isomorphisms.

Finally using proper base change,
and the fact that 
$
\Reali(\psi^\omega_{j+1}(\bar{\bw};\bar{\bx};\cdot)) \neq  \emptyset$ if and only if 
$(\bar{\bw};\bar{\bx}) \in \Reali(\psi^\omega_j)$,
we get that the restriction homomorphisms
\[
\HH^i(\Reali(\psi^\omega_{j}) \times \underbrace{\bbP^{m_{j+1}} \times \cdots  \times \bbP^{m_{j+1}}}_{N_{j+1}})  \rightarrow
\HH^i(\Reali(\psi^\omega_{j+1})) 
\]
are isomorphisms for $0 \leq i \leq 2 d_{j}$, from which it follows using 
\eqref{eqn:complexjoin2.1'} that 
\[
Q({\psi^{\omega}_{j}})= F_{j+1}^\omega(Q({\psi^{\omega}_{j+1}})), 
\]
which completes  the inductive step in this case.
\item
$\omega(j+1) = \forall $.
For each 
$(\bar{\bw};\bar{\bx}) \in \bbP^{\mathbf{m}_j}$
(where 
$\bar{\bw} = (\bw^{(1)};\cdots\; \bw^{(m)} \in \bbP^{e_1}  \times \cdots \times \bbP^{e_m}$,
$\bar{\bx} = (\bar{\bx}_1;\cdots;\bar{\bx}_{j})$,
and for $1 \leq h \leq j$, 
$\bar{\bx}_h = (\cdots; \bx^{(i_1,\ldots,i_{h-1})};\cdots) \in  \underbrace{\bbP^{m_h} \times \cdots  \times \bbP^{m_h}}_{N_h}$),
and each tuple $(i_1,\ldots,i_{j}) \in [0,2d_0+1] \times \cdots \times [0, 2d_{j-1}]$, 
let 
$W^{(i_1,\ldots,i_j)}_{\bar{\bw};\bar{\bx}} =  \bbP^{m_{j+1}} \setminus V^{(i_1,\ldots,i_j)}_{\bar{\bw};\bar{\bx}}$.

Notice that, for 
$0 \leq i \leq 2d_{j}$,
the restriction homomorphism
\[
\HH^i(\bbP^{m_{j+1}}) \rightarrow \HH^i(W^{(i_1,\ldots,i_j)}_{\bar{\bw};\bar{\bx}})
\]
is an isomorphism
using Part~\eqref{itemlabel:thm:cc:2} of Theorem~\ref{the:compactcovering_general}.

Also, observe that denoting by
\[
\pi_{(i_1,\ldots,i_j)}: \underbrace{\bbP^{m_{j+1}} \times \cdots \times \bbP^{m_{j+1}}}_{N_{j+1}} \rightarrow \bbP^{m_{j+1}},
\]
the projection on the $(i_1,\ldots,i_j)$-th factor, 
\[
\Reali(\psi^\omega_{j+1}(\bar{\bw};\bar{\bx};\cdot)) = \bigcup_{(i_1,\ldots,i_{j}) \in [0,2d_0+1] \times \cdots \times [0, 2d_{j-1}]} \pi_{(i_1,\ldots,i_j)}^{-1}(W^{(i_1,\ldots,i_j)}_{\bar{\bw};\bar{\bx}} ).
\]

Now using Part \eqref{itemlabel:lem:product:2} of Lemma~\ref{lem:product} we get that
for each point $(\bar{\bw};\bar{\bx}) \in  \bbP^{\mathbf{m}_{j}} \setminus \Reali(\psi^\omega_j) \subset \bbP^{\mathbf{m}_{j}}$,
and for $0 \leq i \leq 2d_{j}$ the restriction homomorphisms
\[
\HH^i(\underbrace{\bbP^{m_{j+1}} \times \cdots \times \bbP^{m_{j+1}}}_{N_{j+1}}) \rightarrow 
\HH^i(\underbrace{\bbP^{m_{j+1}} \times \cdots \times \bbP^{m_{j+1}}}_{N_{j+1}} - \Reali(\psi^\omega_{j+1}(\bar{\bw};\bar{\bx};\cdot)))
\]
are isomorphisms.

Finally using proper base change,
and the fact that 
\[
\underbrace{\bbP^{m_{j+1}} \times \cdots \times \bbP^{m_{j+1}}}_{N_{j+1}} \setminus \Reali(\psi^\omega_{j+1}(\bar{\bw};\bar{\bx};\cdot)) \neq  \emptyset
\] 
if and only if 
$(\bar{\bw};\bar{\bx}) \in \bbP^{\mathbf{m}_j} - \Reali(\psi^\omega_j)$,
we get that the restriction homomorphisms
\[
\HH^i((\bbP^{\mathbf{m}_j} - \Reali(\psi^\omega_{j})) \times \underbrace{\bbP^{m_{j+1}} \times \cdots  \times \bbP^{m_{j+1}}}_{N_{j+1}})  \rightarrow
\HH^i( \bbP^{\mathbf{m}_{j+1}} \setminus \Reali(\psi^\omega_{j+1})) 
\]
are isomorphisms for $0 \leq i \leq 2 d_{j}$. From this it follows 
using Theorem~\ref{the:alexanderduality_general} twice, and  \eqref{eqn:complexjoin2.2'},  that 
\[
Q({\psi^{\omega}_{j}}) = F_{j+1}^\omega(Q({\psi^{\omega}_{j+1}})), 
\]
which completes  the inductive step in this case.
\end{enumerate}
\end{proof}

\section{An algebraic version of  Toda's theorem over algebraically closed fields}
\label{sec:Toda}
As mentioned previously, an important feature of Theorem~\ref{thm:qe} (and Corollary~\ref{cor:qe}) is that the 
quantifier-free  formula $J_{m,n}(\psi)$ obtained from the quantified formula $\psi$ has an easy description in terms of 
$\psi$ (in contrast to what happens in classical quantifier elimination). Making this statement quantitative leads
to a result which is formally analogous to a classical result in discrete complexity theory -- namely, Toda's theorem.

\subsection{The classes $\bP^c_k$, $\mathbf{PH}_k^c$, $\#\bP^c_k$}
\label{subsec:complexity-classes}
We fix an algebraically closed field $k$ for the rest of this section.
In order to prove our algebraic analog of Toda's theorem we first need 
algebraic analogs of the complexity classes appearing in Toda's theorem.
In order to motivate the definition of the polynomial hierarchy 
it is instructive to first consider the following set-theoretic definitions.

Recall that any map $f:X \rightarrow Y$ between sets $X$ and $Y$  induces three functors 

\[
\textbf{Pow}(X) {\xrightarrow{f_\exists} \atop {\xleftarrow{f^*}  \atop \xrightarrow{f_\forall}}}  \textbf{Pow}(Y).
\]
in the poset categories of their respective power sets $\textbf{Pow}(X), \textbf{Pow}(Y)$.
The functors
$f^*,f^\exists,f^\forall$  are defined as follows.  For all 
$A \in \Ob(\textbf{Pow}(X))$ and
$B \in \Ob(\textbf{Pow}(Y))$,  
\begin{eqnarray*}
f^*(B) & =&  f^{-1}(B), \\
f_\exists(A) &=& \{y \in Y \mid (\exists x \in X)((f(x) = y) \wedge (x \in A)) \},\\
f_\forall(A)  &=&  \{ y \in Y \mid (\forall x\in X) ( (f(x) = y)  \implies (x \in A)) \}.
\end{eqnarray*}

 Now, suppose that $X = \bbP^n \times \bbP^m$, $Y = \bbP^m$ and $\pi:\bbP^m \times \bbP^n \rightarrow \bbP^n$.
Let $V$ be an algebraic subset of $X$. Then, 
$\pi_\exists (V), \pi_\forall(V)$ are both algebraic subsets  of $\bbP^n$.

However, as is well known from computational algebraic geometry, elimination is a costly procedure, and as a result the 
`complexity'  of $\pi_\exists(V)$ and $\pi_\forall(V)$ could increase dramatically compared to that of $V$. Here, by complexity one can take 
for instance the number and degrees of the polynomials appearing in the descriptions of these sets. A more precise definition of complexity and formalization in terms of sequences of algebraic sets rather than just one, leads to variants of the famous $\bP$ vs $\bN\bP$ (respectively, $\bP$ vs co-$\bN\bP$) question albeit over the field $k$ \cite{BSS}. 
Alternating the functors $\pi_\exists,\pi_\forall$ a  fixed number of times leads to the so called polynomial hierarchy of complexity classes whose lowest level consists of the class $\bP$ of sequences of objects with polynomially bounded growth in complexity.
We now  make more precise the notion of `complexity'  that we are going to use. We begin with some notation.

\begin{notation}  
For any finite tuple $\bn=(n_1,\ldots,n_m) \in \bbN^{m}$, we denote:
\begin{enumerate}
\item $\bn^{(j)} = (n_{j+1},\ldots,n_m)$ for $0 \leq j < m$  
(we  will denote $\bn' = \bn^{(1)}$ for convenience);
\item
$\pi_{\bn,j}: \bbP^\bn\rightarrow \bbP^{\bn^{(j)}}$, the projection map.
\end{enumerate}
\end{notation}

\begin{definition}[Complexity of algebraic sets and polynomial maps]
\label{def:complexity}
Following \cite{Isik}, we define the complexity, $c(V)$,  of an algebraic subset $V \subset \bbP^\bn$, to be the size of the smallest arithmetic circuit  \cite{Burgisser-book}
computing a tuple of multi-homogeneous polynomials $(f_1,\ldots,f_s)$  such that $V = Z(f_1,\ldots,f_s)$.
The complexity $c(g)$ of a polynomial map $g: \bbZ^m \rightarrow \bbZ^n$ is the size of the smallest arithmetic
circuit computing $g$. 
\end{definition}

\begin{remark}
We will often  identify for convenience  $\bbZ^m$ with the $\bbZ$-module, $\bbZ[T]_{\leq m-1}$,
of polynomials of degree at most $m-1$.
\end{remark}

\begin{notation}[Characteristic function]
Let $L = (V_i \subset \bbP^{\bn_i})_{i \in \bbN}$ be a tuple indexed by $\bbN$, where each
$V_i$ is an algebraic subset of $\bbP^{\bn_i}$.

We will denote by $\mathbf{1}_L$ the tuple of constructible functions 
\[
(1_{V_i}:\bbP^{\bn_i} \rightarrow \{0,1\} \subset \bbZ \subset \bbZ[T])_{i \in \bbN},
\]
where $1_V$ denotes the characteristic function of $V(k)$.
\end{notation}

\begin{definition}[The class $\bP^c_k$ and $\bP_\bbZ$]
\label{def:P}
Following \cite{Isik}, we will say that 
\[
L = (V_i \subset \bbP^{\bn_i})_{i \in \bbN} \in \bP^c_k
\]
if 
$c(V_i), |\bn_i|$ are polynomially bounded functions of $i$. 
Similarly, we will say that a sequence
$G = (g_i: \bbZ^{m_i} \rightarrow \bbZ^{n_i})_{i \in \bbN} \in \bP_{\bbZ}$ if 
$c(g_i), m_i,n_i$ are all polynomially bounded functions of $i$. 
\end{definition}

\begin{example}
For each fixed $d$, consider the sequence 
\[
L_d = \left(V_m \subset \bbP^{\bn_m}\right)_{m \in \bbN},
\]
where 
\[
\bn_m = \left(m, \underbrace{\binom{m+d}{d},\cdots, \binom{m+d}{d}}_{m+1}\right),
\]
and 
\[
V_m = \{(x,f_0,\ldots,f_{m}) \mid  f_i(x)= 0, 0 \leq i \leq m\},
\]
where we identify $\bbP^{\binom{m+d}{d}}$ with the projectivization of the space of non-zero homogeneous polynomials of degree $d$ in $m+1$ variables.
It is an easy exercise to check that $L_d \in \bP^c_k$ for each $d \geq 0$.
\end{example}

We now define the algebraic analog  $\#\bP^c_k$  of the discrete complexity class $\#\bP$. Note that in the classical theory the class 
$\#\bP$ consists of `counting functions' counting the number of solutions of the `fibers'  of some 
Boolean satisfiability problem belonging to $\bP$.  As remarked 
before, a natural analog of counting in the algebraic context is computing  the Poincar\'e polynomial of algebraic sets  (or some easily computable polynomial function of the Poincar\'e polynomial). Thus, it is natural to
define the algebraic analog of  $\#\bP$ as sequences of constructible functions whose values are the Poincar\'e polynomials
(with respect to etale cohomology)
of the fibers of sequences of proper morphisms.
The sequence of codomains of the morphisms defining an element of the class  $\#\bP^c_k$ should itself belong  to 
$\bP_k^c$. 

More formally, 
we define:

\begin{definition}[The class $\#\bP^c_k$]
\label{def:sharp-P}
A sequence $F = (F_i: \bbP^{\bn_i} \rightarrow \bbZ^{N_i})_{i \in \bbN}$,
where each $F_i$ is a constructible function, is in the class $\#\bP^c_k$, 
if and only if there exists
\[
L = \left(V_i \subset \bbP^{\mathbf{m}_i}\right)_{i \in \bbN} \in \bP^c_k,
\] 
\[
\bj: \bbN \rightarrow \bbN,
\] 
and 
\[
\left(g_i: \bbZ^{2(|\mathbf{m}_i|- |\mathbf{m}^{(\bj(i))}_i|)+1}  \rightarrow \bbZ^{N_i}\right)_{i \in \bbN} \in \bP_\bbZ,
\]
such that for all $i \in \bbN$,
\[
F_i (\bz) = g_i(P_{\pi_{\mathbf{m}_i,\bj(i)}^{-1}(\bz)}).
\]
\end{definition}

\begin{notation}[$\exists L$ and $\forall L$]
For a tuple $L = (V_i \subset \bbP^{\bn_i})_{i \in \bbN}$  of algebraic subsets  of $\bbP^{\bn_i}$,
we denote by 
\[
\exists L := (\pi_{\bn_i,1}(V_i) \subset \bbP^{\bn_i'})_{i \in \bbN},
\] 
and 
\[
\forall L :=  (\pi_{\bn_i,1,\forall}(V_i) \subset \bbP^{\bn_i'}_k)_{i \in \bbN} =
(\bbP^{\bn_i'} - \pi_{\bn_i,1}(\bbP^{\bn_i} \setminus V_i) \subset \bbP^{\bn_i'})_{i \in \bbN}.
\]
\end{notation}

\begin{definition}[Polynomial hierarchy]
\label{def:PH}
For $i \geq 0$, we define $\Pi^{c,i}_k,\Sigma^{c,i}_k$ as follows.
\begin{enumerate}
\item
$\Pi^{c,0}_k = \Sigma^{c,0}_k = \bP^c_k$;
\item
For $i >0$,
we define $\Sigma^{i+1,c}_k$ as the smallest class of sequences $L = (V_i \subset \bbP^{\bn_i}_k)_{i \in \bbN}$ satisfying:
\begin{enumerate}
\item
$\Pi^{i,c}_k \subset \Sigma^{i+1,c}_k$, and
\item
$L \in \Sigma^{i+1,c}_k \Longrightarrow \exists L \in \Sigma^{i+1,c}_k$.
\end{enumerate}
 
\item
Similarly, we define 
$\Pi^{i+1,c}_k$ as the smallest class of sequences $L = (V_i \subset \bbP^{\bn_i}_k)_{i \in \bbN}$ satisfying:
\begin{enumerate}
\item
$\Sigma^{i,c}_k \subset \Pi^{i+1,c}_k$, and 
\item
$L \in \Pi^{i+1,c}_k \Longrightarrow \forall L \in \Pi^{i+1,c}_k$.
\end{enumerate}
\item
Finally, we define 
\[
\mathbf{PH}_k^c = \bigcup_{ i \geq 0} \left(\Pi^{i,c}_k \cup \Sigma^{i,c}_k\right),
\]
and 
\[
\mathbf{1}_{\mathbf{PH}_k^c} = \{ \mathbf{1}_L: L \in \mathbf{PH}_k^c\}.
\]
\end{enumerate}
\end{definition}

\begin{remark}
\label{rem:PH}
Notice that it follows from Definition~\ref{def:PH} that $L \in \mathbf{PH}_k^c$ if and only if there exists $L' \in \mathbf{P}_k^c$, 
$n \geq 0$, and $\bQ_1,\ldots,\bQ_n \in \{\exists,\forall\}$, such that 
\[
L = \bQ_1 \cdots \bQ_n L'.
\]
\end{remark}

With the algebraic analogs of the classes $\#\bP$, and $\bP\bH$ in place (cf. Definitions~\ref{def:sharp-P} and \ref{def:PH} respectively), we are now in a position to state an algebraic analog of Toda's theorem.

\begin{theorem}[Algebraic analog of Toda's theorem]
\label{thm:Toda-algebraic}
\[
\mathbf{1}_{\mathbf{PH}_k^c} \subset \#\mathbf{P}^c_k.
\]
\end{theorem}

\subsection{Proof of 
algebraic version of  Toda's theorem}

\begin{lemma}
\label{lem:polynomiality-of-J}
Let $L = (V_i \subset \bbP^{\mathbf{m}_i} \times \bbP^{\bn_i})_{i \in \bbN} \in \bP_k^c$, with 
$\mathbf{m}_i = (e_{i,1},\ldots,e_{i,m_i}) \in \bbN^{m_i},
\bn_i = (f_{i,1},\ldots,f_{i,n}) \in \bbN^{n}$. 
Then, 
\[(
J_{m_i,n}(V_i))_{i \in \bbN} \in \bP_k^c.
\]
\end{lemma}

\begin{proof}
First observe that it follows from Definitions~\ref{def:P} and \ref{def:complexity}  that for each $i \in \bbN$ there exists a tuple 
$\bar{f}_i = (f_{i,1},\ldots, f_{i,k_i})$ of multi-homogeneous polynomials  such that there exists an arithmetic circuit computing $\bar{f}_i$ of 
size $C_i$ which is polynomially bounded in $i$, and such that
$V_i$ is defined by the  proper quantifier-free formula
\[
\psi_i \defeq \bigwedge_{j=1}^{k_i} (f_{i,j} = 0).
\]

It now follows from Notation~\ref{not:generalized-join} that 
\begin{enumerate}
\item 
$
J_{m_i,n}(\psi_i) = \bigwedge_{j=1}^{K_i} \psi_{i,j},
$ where
\item
$K_i = 2^{n}\prod_{j=1}^{n} (d_{i,j-1}+ 1)$, and
$d_{i,0},\ldots d_{i,n-1}$ are defined as in Notation \ref{not:generalized-join};
\item
for each $j \in [1,n]$, the sequence $(d_{i,j-1})_{i \in \bbN}$ is polynomially bounded in $i$;
\item
for each $i,j$, $\psi_{i,j} = \bigwedge_{h=1}^{k_i} (F_{i,j,h} = 0)$ for some multi-homogeneous polynomials $F_{i,j,h}$, and
\item 
there exists an arithmetic circuit  of size $C_{ij}$
computing the tuple 
\[
(F_{i,j,1},\ldots,F_{i,j,k_i}),
\] 
and  for each $j \in [1,n]$, the sequence $(C_{i,j})_{i \in \bbN}$ is polynomially bounded in $i$.
\end{enumerate}

This shows that 
\[
c(J_{m_i,n}(V_i)) \leq 2^{n}\prod_{j=1}^{n} (d_{i,j-1} +1) C_{i j},
\]
and hence the sequence $(c(J_{m_i,n}(V_i)))_{i \in \bbN}$ 
is polynomially bounded in $i$, since $n$ is a constant,  and 
the sequences $(d_{i,j})_{i \in \bbN}$ and $(C_{ij})_{i \in \bbN}$ are bounded polynomially in $i$, as observed previously.
This proves the lemma.
\end{proof}

\begin{lemma}
\label{lem:complexity-of-operators}
The following sequences belong to $\bP_{\bbZ}$.
\begin{enumerate}
\item
$(\mathrm{Rec}_{\bn_i}: \bbZ[T]_{\leq |\bn_i|} \rightarrow \bbZ[T]_{\leq |\bn_i|})_{i \in \bbN}$, for any sequence
$(\bn_i)_{i \in \bbN}$ such that the sequence $(|\bn_i|)_{i \in \bbN}$ is polynomially bounded.
\item
$(\mathrm{Trunc}_{m_i,n_i}: \bbZ^{n_i+1} \rightarrow \bbZ^{m_i+1})_{i \in \bbN}$, for any pair of polynomially bounded 
sequences $(m_i)_{i \in \bbN}, (n_i)_{i \in \bbN}$;
\item
$(M_{(1-T)^{N_i}}: \bbZ[T]_{d_i} \rightarrow \bbZ[T]_{d_i+N_i})_{i \in \bbN}$,
where $(d_i), (N_i)_{i \in \bbN}$ are two  polynomially bounded sequences,
for $f \in \bbZ[T]$, $M_f(g) = f g$;
\item
$(\mathrm{pseudo}_{n_i} : \bbZ[T]_{\leq 2 n_i} \rightarrow \bbZ[T]_{n_i})_{i \in \bbN}$, for any polynomially bounded sequence 
$(n_i)_{i \in \bbN}$, where $\mathrm{pseudo}$ maps a polynomial $P(T)$ to $ P^{{\mathrm{even}}} - T P^{{\mathrm{odd}}}$.
\end{enumerate} 
\end{lemma}

\begin{proof}
Obvious.
\end{proof}

\begin{proof}[Proof of Theorem~\ref{thm:Toda-algebraic}]
Suppose that $L = (V_i \subset \bbP^{\mathbf{m}_i})_{i \in \bbN} \in \bP\bH_k^c$. It follows from Remark~\ref{rem:PH} that there
exists $n \geq 0$
$L'  = (V_i' \subset \bbP^{\mathbf{m}_i} \times \bbP^{\bn_i})_{i \in \bbN} \in \bP_k^c$, 
with $\mathbf{m}_i \in \bbN^{m_i}, \bn_i \in \bbN^{n}$ for some fixed $n$,
and $\bQ_1,\ldots,\bQ_n \in \{\exists,\forall\}$, such that 
\[
L = \bQ_1 \cdots \bQ_n L'.
\]
This implies that for each $i \in \bbN$,
$V_i = (V_i')^{\omega_i}$, 
where $\omega_i \in \{\exists,\forall\}^{[1,n]}$, is defined by $\omega_i(j) = \bQ_j$.
Lemma~\ref{lem:polynomiality-of-J} now implies that $(J_{m_i,n}(V_i'))_{i \in \bbN} \in \bP_k^c$.

Let $\pi_{\mathbf{m}_i,\bn_i}: V_i' \rightarrow \bbP^{\mathbf{m}_i}$ 
(respectively, $J(\pi_{\mathbf{m}_i,\bn_i}): J_{m_i,n}(V_i') \rightarrow \bbP^{\mathbf{m}_i}$)  
denote the restriction of the projection morphism
to $V_i'$ (respectively, $J_{m_i,n}(V_i')$). 

Let $\pi_{\mathbf{m}_i,\bn_i,\bw: }V'_{i,\bw} \rightarrow \{\bw\}$ 
(respectively, $J(\pi_{\mathbf{m}_i,\bn_i,\bw}): J_{m_i,n}(V_i')_{\bw} \rightarrow \bbP^{\mathbf{m}_i}$)  
denote the pull-back of $\pi_{\mathbf{m}_i,\bn_i}$ 
(respectively, $J(\pi_{\mathbf{m}_i,\bn_i})$)
under the inclusion $\{\bw\} \hookrightarrow \bbP^{\mathbf{m}_i}$.

Observe that,
\[
(J_{m_i,n}(V'_i))_{\bw} \cong J_{0,n}(V'_{i,\bw}).
\]

Theorem~\ref{thm:qe} now implies that
\[
1_{V_i} =  F^{\omega_i}(Q({J_{0,n}(V'_{i,\bw})})) =  F^{\omega_i} \circ \mathrm{pseudo}_{d_{i,n}} (P({J_{0,n}(V'_{i,\bw})})),
\]
where 
$F^{\omega_i}$ is the operator appearing in Theorem~\ref{thm:qe}.
 
It follows also from the definition of the operator $F^{\omega_i}$ (as in Theorem~\ref{thm:qe}) and 
Lemma~\ref{lem:complexity-of-operators}, that the 
two sequences of operators $(F^{\omega_i})_{i \in \bbN} \in \bP_{\bbZ}$, 
$(\mathrm{pseudo}_{d_{i,n}})_{i \in \bbN}$ are in $\bP_{\bbZ}$, and so is the sequence of their compositions. 
It now follows 
from Definition \ref{def:sharp-P}  that 
the sequence $(1_{V_i})_{i \in \bbN} \in \#\bP^c_k$.
\end{proof}

\section{Bounds on Betti numbers}
\label{sec:Katz}
As before, we work over an algebraically closed field $k$. We fix a prime number $\ell \neq  \Char(k)$, and work with etale cohomology with $\bbQ_{\ell}$-coefficients. Let $X \subset \bbP^{m} \times \bbP^{n}$ be an algebraic subset.
In this section, we will apply the results of the previous section to obtain bounds on sums of the Betti numbers 
of the image $\pi(X)$ under the projection to $\bbP^{m}$ in terms of those of the relative join. Finally, we compare
this bound with those achieved through an application of classical elimination theory.

\subsection{Classical results on bounds for sums of Betti numbers of algebraic sets}
\label{subsec:classical}
In this subsection, we recall some classical results on bounds of (sums of) Betti numbers for algebraic subsets of $\bbA^n$ and $\bbP^n$. The results here are due to Ole{\u\i}nik and Petrovski{\u\i}, 
Thom,
Milnor, Bombieri, Adolphson-Sperber, and Katz. We follow closely the paper of Katz (\cite{K1}).\\

Given an algebraic set  $X$, let 
\[
h^{i}(X) := \dim(\rH^{i}(X,\bbQ_{\ell}))
\]
 (resp. $h^{i}_c(X) : = \dim(\rH^{i}_c(X,\bbQ_{\ell})$)). 
Let $h(X) = \sum_i h^{i}(X)$ and $h_c(X) = \sum_i h^{i}_c(X)$. 
Finally, we denote by $\chi(X)$ and $\chi_c(X)$ the Euler characteristic (resp. compactly supported Euler characteristic) of $X$. With this notation, one has the following classical bounds on sums of Betti numbers and Euler characteristics.\\

\begin{enumerate}[(1)]
\item 
\label{itemlabel:Betti:1}
Suppose $ \Char(k) = 0$. If $X \subset \bbA^{n}$ ($n \geq 1$) defined by $
r \geq 1$ equations $F_i$ with $deg(F_i) \leq d$, then 
Ole{\u\i}nik and Petrovski{\u\i} \cite{OP}, 
Thom \cite{T} and
Milnor \cite{Milnor2}
showed that $$h(X) \leq d(2d-1)^{2n-1}.$$ While the result in \emph{loc. cit.} is stated for singular cohomology with coefficients in $\bbQ$, standard arguments give the same result for $\ell$-adic cohomology over any algebraically closed field of characteristic zero. Standard arguments (\cite{K1}) now show that $$h_c(X) \leq 2^r(1+rd)(1+2rd)^{2n+1}.$$
\item 
\label{itemlabel:Betti:2}
In general, Bombieri~\cite{Bomb} gave the explicit upper bound $$|\chi_c(X)| \leq (4(1+d)+5)^{n+r}.$$ 
\item 
\label{itemlabel:Betti:3}
Bombieri's bounds were improved upon by Adolphson and Sperber (\cite{AS}). They considered the homogeneous polynomial $$D_{n,r}(X_0,\ldots,X_n):= \Sigma_{|W| =n} X^{W},$$ and showed that 
$$|\chi_c(X) \leq 2^rD_{n,r}(1,1+d,1+d,\ldots,1+d) \leq 2^{r} (r+1+rd)^{n}.$$
\item 
\label{itemlabel:Betti:4}
In~\cite{K1}, Katz derived bounds on sums of Betti numbers given any universal bound $$|\chi_c(X)| \leq E(n,r,d).$$
More precisely, let 
$$A(n,r,d) : = E(n,r,d) + 2 + 2 \sum_{i=1}^{n-1} E(i,r,d),$$
and
$$B(n,r,d):= 1 + \sum_{\emptyset \neq S \subset \{1,2,\ldots,r\}} A(n+1,1,1+d(\#S)).$$
Then for $X$ as before, Katz showed~\cite[Theorem 1]{K1} that $$h_c(X) \leq B(n,r,d).$$
\item 
\label{itemlabel:Betti:5}
Suppose now that $X \subset \bbP^{n}$ is defined by the vanishing of $r \geq 1$ homogeneous polynomials of degree at most $d$. 
Then~\cite[Theorem 3]{K1} gives:
$$h_c(X) = h(X) \leq 1 + \sum_{i=1}^{n} B(i,r,d).$$
\end{enumerate}

Here are some explicit versions of this bound. 

\begin{enumerate}
\item Bombieri's bound, gives 
$$ B(n,r,d) \leq 2^{r} \times (5/4) \times (4(2+rd+5)^{n+2}.$$
\item The Adolphson-Sperber bound gives 
$$ B(n,r,d) \leq 2^r \times 3 \times 2 \times (2+(1+rd))^{n+1}.$$
\end{enumerate}

In particular, one has the following bounds due to Katz:

\begin{enumerate}
\item For $X \subset \bbA^{n}$ defined by $r$ polynomials of degree $\leq d$, the Adolphson-Sperber bound gives:
$$h_c(X) \leq 2^r \times 3 \times 2 \times (2+(1+rd))^{n+1}.$$
\item For $X \subset \bbP^{n}$ defined by $r$ homogeneous polynomials of degree $\leq d$, the Adolphson-Sperber bound gives:
$$h_c(X) = h(X) \leq (3/2) \times 2^r \times 3 \times 2 \times (2+(1+rd))^{n+1}.$$
\end{enumerate}

We can apply these results to obtain bounds on sums of the Betti numbers for $X \subset \bbP^{n} \times \bbP^{m}$ defined by a bi-homogeneous system
$F_i = F_i(X_0,\ldots,X_{n},Y_0,\ldots,Y_{m})$ with bi-homogeneous degree bounded by $(d_1,d_2)$. The above bounds then give the following:

\begin{proposition}
\label{prop:katz-for-products}
Let $X \subset \bbP^{n} \times \bbP^{m}$ be an algebraic set defined $r$ by bi-homogeneous polynomials $ F_i(X_0,\ldots,X_{N},Y_0,\ldots,Y_{M})$ of bi-degree $(d_1,d_2)$. Then one has:
\begin{equation*}
h_c(X) = h(X) \leq \sum_{\substack{0\leq  i \leq n \\ 0 \leq j \leq m}} B(r,d_1+d_2,i+j).
\end{equation*}
Here, for $i+j = 0$, we set $B(r,d_1 + d_2,0) = 1$.
\end{proposition}
\begin{proof}
We may decompose $\bbP^{n} \times \bbP^{m} = (\bbA^{n} \times \bbP^{m}) \coprod (\bbP^{n-1} \times \bbP^{m})$. This gives a decomposition 
$X = (X \cap (\bbA^{n} \times \bbP^{m})) \coprod (X \cap (\bbP^{n-1} \times \bbP^{m}))$. One now argues recursively. 
\end{proof}

\subsection{Bounds on the Betti numbers of images via  relative joins}
As a direct consequence of Proposition~\ref{prop:katz-for-products} and  
Theorem~\ref{thm:poincare} we obtain:

\begin{theorem}
\label{thm:katz}
Let $X \subset \bbP^{n} \times \bbP^{m}$ be an algebraic subset defined $r$ by bi-homogeneous polynomials $ F_i(X_0,\ldots,X_{n},Y_0,\ldots,Y_{m})$ of bi-degree $(d_1,d_2)$, and $\pi:\bbP^{n} \times \bbP^{m} \to \bbP^{m}$ the projection morphism.
Then, for all $p > 0$,
\begin{eqnarray*}
\sum_{h=0}^{p-1} b_h(\pi(X)) &\leq & \frac{2}{p} \sum_{h=0}^{p-1} b_h(\rJ^{[p]}_\pi(X))\\
						&\leq&  \frac{2}{p} \sum_{\substack{0\leq  i \leq (n+1)(p+1) -1\\ 0 \leq j \leq m}}B(i+j,r(p+1),d_1+d_2).
\end{eqnarray*}
\end{theorem}

\begin{proof}
The first inequality follows from 
Theorem~\ref{thm:poincare}, and the second from Proposition~\ref{prop:katz-for-products}.
\end{proof}

\section{Relative joins versus products}
\label{sec:comparison}
In Section~\ref{sec:Katz}, upper bounds on the Betti numbers of $\pi(X)$, where $X \subset \bbP^N \times \bbP^n$ is an algebraic subset and $\pi: \bbP^N \times \bbP^n \rightarrow \bbP^n$ were derived in terms of the join $J_\pi^p(X)$.
There is another more direct way to obtain an upper bound on $\pi(X)$:  namely from the spectral sequence associated
to the hypercover

\[
\begin{tikzcd}
  X
    & X\times_\pi X \arrow[l, shift left]
        \arrow[l, shift right]
& X \times_\pi X \times_\pi X  \arrow[l]
\arrow[l, shift left=2] \arrow[l, shift right=2]
        & \cdots 
        \arrow[l, shift left=1] \arrow[l, shift right=1]
\arrow[l, shift left=3] \arrow[l, shift right=3] 
\end{tikzcd}
\]

one obtains the inequality for each $i \geq 0$
\begin{equation}
\label{eqn:hypercover}
b_i(\pi(X)) \leq \sum_{p + q = i} b_q(\underbrace{X \times_\pi \cdots \times_\pi X}_{(p+1)}).
\end{equation}

In this section we compare the upper bounds on Betti numbers coming from considering the relative join with that
coming from inequality \eqref{eqn:hypercover}.

\subsection{Exponentially large error for the hypercovering inequality}
\label{subsec:exponential}
Let $X \subset \bbP^m \times \bbP^n$ and $\pi: \bbP^m \times \bbP^n \rightarrow \bbP^n$ the projection.
Then, for each $p \geq 0$,  it follows from Theorem~\ref{thm:poincare} that
\[
P({\pi(X)}) = (1 - T^2) P({\rJ^{[p]}_\pi(X)}) \mod T^p,
\]
from which it follows that
\begin{equation}
\label{eqn:join}
b_i(\pi(X)) = b_i(\rJ^{[p]}_\pi(X)) - b_{i-2}(\rJ^{[p]}_\pi(X)), 0 \leq i < p.
\end{equation}

Telescoping Eqn. \eqref{eqn:join} we obtain for all odd $p>0$,
\begin{eqnarray}
\label{eqn:telescope-even}
\sum_{2 i < p} b_{2i}(\pi(X)) &=& b_{p-1}(\rJ^{[p]}_\pi(X)), \\
\label{eqn:telescope-odd}
\sum_{2i -1< p} b_{2i -1}(\pi(X)) &=& b_{p-2}(\rJ^{[p]}_\pi(X)).
\end{eqnarray}

Inequalities \eqref{eqn:telescope-even} and  \eqref{eqn:telescope-odd} sometime give more information on the Betti numbers of
$\pi(X)$ than what can be inferred from inequality \eqref{eqn:hypercover}.

For instance, consider the projection map $
\bbP^1 \times \bbP^n \rightarrow \bbP^n$, and $X = \bbP^1 \times \bbP^n$. Applying inequality
\eqref{eqn:hypercover} one gets

\begin{eqnarray*}
1 &=& b_{2n}(\pi(X))\\
 &=& b_{2n}(\bbP^n) \\
&\leq& \sum_{p+q = 2n} b_q(\underbrace{X \times_\pi \cdots \times_\pi X}_{(p+1)})\\
&=&
\sum_{p+q=2n} b_q(\underbrace{\bbP^1_k \times \cdots \times \bbP^1}_{(p+1)} \times \bbP^n)\\
&=&
\sum_{p+q = n} \sum_{0 \leq j \leq 2q} \binom{2p+1}{j} \\
&=&\sum_{0 \leq p \leq n} \sum_{0 \leq j \leq 2(n-p)} \binom{2p+1}{j}.
\end{eqnarray*}
This example shows that the difference between the two sides of the inequality \eqref{eqn:hypercover} 
can be exponentially large in $n$.

On other hand, it follows from the fact that $\rJ^{[2n+1]}_\pi(X) = \bbP^{2(2n+2)-1} \times \bbP^n$, and Eqn. 
\eqref{eqn:telescope-even}, that with $p=2n+1$
\begin{eqnarray*}
\sum_{2 i < p} b_{2i}(\pi(X)) &=& b_{2n}(\rJ^{[2n+1]}_\pi(X)), \\
&=&b_{2n}(\bbP^{2(2n+2)-1}_k \times \bbP^n_k) \\
&=& n+1.
\end{eqnarray*}

\subsection{Joins and defects}
We discuss another way in which the relative join gives better information on the Betti numbers of the image under
projection of an algebraic set than what can be gleaned from inequality \eqref{eqn:hypercover}. We prove the following theorem.

\begin{theorem}
\label{thm:join-defect}
Let $X \subset \bbP^N \times \bbP^n$ be a subvariety  defined by $N+r$ bi-homogeneous forms. Let $\pi: \bbP^N \times \bbP^n \rightarrow \bbP^n$ be the projection morphism. Then, for all $i, 0 \leq i < \lfloor \frac{n-r}{r} \rfloor$,
\begin{eqnarray*}
b_i(\pi(X)) &=& 1 \mbox{ if $i$ is even}, \\
b_i(\pi(X)) &=& 0 \mbox{ if $i$ is odd}.
\end{eqnarray*}
\end{theorem}

We first need a preliminary result.

\begin{lemma}\label{lem:singularWL}
Let $Y : =  \bbP^{a} \times \bbP^b $ and $X \subset Y$ be a closed subvariety defined by $r$-bi-homogeneous forms. Then the natural restriction restriction map on cohomology
$$\rH^{i}(Y) \rightarrow \rH^{i}(X)$$
is an isomorphism for all $i < \dim(Y) - r$. 
\end{lemma}
\begin{proof}
Suppose $r =1$. Then, since the complement of the zeros of a bi-homogeneous form in $Y$ is an affine variety, the result follows from usual Artin vanishing for affine schemes. In general, the complement of $X$ is covered by a union of $r$ affine open sets. One can now argue as in (\cite[3.2]{GL}).
\end{proof}

\begin{proof}[Proof of Theorem~\ref{thm:join-defect}]
For any $p \geq 0$, $\rJ^{[p]}_\pi(X)$ is an algebraic subset of $\bbP^{(p+1)(N+1)-1} \times \bbP^n$. Thus the ambient dimension, $M$, of 
$\rJ^{[p]}_\pi(X)$ equals  $(p+1)N +p + n$.
Since $X$ is defined by $N+r$ equations,
it follows that the number of equations $E$  needed to define $\rJ^{[p]}_\pi(X)$  is $(p+1)(N+r)$.

Using Lemma \ref{lem:singularWL}, we deduce that for
\begin{eqnarray*}
0 \;\leq\; i &<&  \dim \rJ^{[p]}_\pi(X) - ( E - \mathrm{codim}\rJ^{[p]}_\pi(X))\\
& = & M -E \\
&=& (p+1)N + p +n - (p+1)(N+r) \\
&=& p + n - (p+1)r \\
&=&  n -r - p(r-1),
\end{eqnarray*}
we have 
\[
b_i(\rJ^{[p]}_\pi(X)) = b_i(\bbP^{(p+1)(N+1)-1}\times \bbP^n).
\]

On other hand
\[
b_i(\pi(X)) = b_i(\rJ^{[p]}_\pi(X)) - b_{i-2}(\rJ^{[p]}_\pi(X)),
\]
for $0 \leq i < p$.
It follows that for $0 \leq i < \min(p, n -r - p(r-1))$,
\[
b_i(\pi(X)) = b_i(\bbP^{(p+1)(N+1)-1} \times \bbP^n) - b_{i-2}(\bbP^{(p+1)(N+1)-1} \times \bbP^n).
\]

The integral value of $p$ that maximizes the function  $\min(p, n -r - p(r-1))$ equals 
$\lfloor\frac{n-r}{r}\rfloor $ from which we deduce that
for $0 \leq i < p_0=  \lfloor\frac{n-r}{r}\rfloor  $,
\begin{equation}
\label{eqn:defect}
b_i(\pi(X)) = b_i(\bbP^{(p_0+1)(N+1)-1} \times \bbP^n) - b_{i-2}(\bbP^{(p_0+1)(N+1)-1} \times \bbP^n).
\end{equation}
The theorem follows from \eqref{eqn:defect}.
\end{proof}

\bibliographystyle{amsalpha}

\bibliography{mybib}

\end{document}